\documentclass[11pt,leqno,twoside]{amsart}

\usepackage[utf8]{inputenc}
\usepackage{eucal}

\usepackage[usenames]{color}

\usepackage[normalem]{ulem}
\usepackage{epigraph}	
\usepackage[T1]{fontenc}
\usepackage{enumerate, comment}
\usepackage{mathrsfs}
\usepackage{amsmath, amssymb, amsthm, amsfonts}
\usepackage{url, doi, hyperref}
\usepackage{pifont, mathbbol, dsfont}
\usepackage{tikz, subfigure, xcolor, graphicx} 
\usepackage{pgfplots}
\usepackage{cite}

\newtheorem{theorem}{Theorem}[section]
\newtheorem{lemma}[theorem]{Lemma}
\newtheorem{proposition}[theorem]{Proposition}
\newtheorem{corollary}[theorem]{Corollary}
\newtheorem{assumption}[theorem]{Assumption}
\theoremstyle{definition}
\newtheorem{definition}[theorem]{Definition}
\newtheorem{remark}[theorem]{Remark}

\numberwithin{equation}{section}

\providecommand{\norm}[1]{\lVert#1\rVert} %Norm
\providecommand{\abs}[1]{\lvert#1\rvert} % absolut value

\DeclareMathOperator{\Ran}{Ran}

\newcommand{\au}{\underline{a}}

\newcommand{\tu}{\underline{t}}
\newcommand{\su}{\underline{s}}
\newcommand{\gu}{\underline{g}}

\newcommand{\R}{\mathbb{R}}
\newcommand{\T}{\mathbb{T}}

\newcommand{\C}{\mathbb{C}}
\newcommand{\Z}{\mathbb{Z}}
\newcommand{\N}{\mathbb{N}}

\newcommand{\B}{\mathbb{B}}

\newcommand{\cH}{{\mathcal H}}

\newcommand{\cL}{{\mathcal L}}

\newcommand{\cR}{{\mathcal R}}
\newcommand{\cS}{{\mathcal S}}

\newcommand{\cV}{{\mathcal V}}
\newcommand{\cX}{{\mathcal X}}

\newcommand{\Me}{{\mathcal M}}

\newcommand{\Ke}{{\mathcal K}}
\newcommand{\Ge}{{\mathcal G}}
\newcommand{\Ie}{{\mathcal I}}
\newcommand{\Ee}{{\mathcal E}}

\newcommand{\Be}{{\mathcal B}}

\newcommand\restr[2]{{% we make the whole thing an ordinary symbol
 \left.\kern-\nulldelimiterspace % automatically resize the bar with \right
 #1 % the function
 \vphantom{\big|} % pretend it's a little taller at normal size
 \right|_{#2} % this is the delimiter
 }}

\usepackage{eucal}
\newcommand{\verts}{{\mathcal V}}

\def\mv{\mathsf{v}}
\def\me{\mathsf{e}}

 \DeclareMathOperator{\Real}{Re}

\title[Time-graphs and evolution equations]{If time were a graph, what would\\ evolution equations look like?}

\author[Hussein]{Amru Hussein} 
\address{Amru Hussein, Department of Mathematics,
	TU Kaiserslautern,  Paul-Ehrlich-Stra\ss e 31, 67663 Kaisers\-lautern, Germany}
\email{hussein@mathematik.uni-kl.de}

\author[Mugnolo]{Delio Mugnolo}
\address{Delio Mugnolo, Chair of Analysis, Faculty of Mathematics and Computer Science, University of Hagen, 58084 Hagen, Germany}
\email{delio.mugnolo@fernuni-hagen.de}

\thanks{The second author was partially supported by the Deutsche Forschungsgemeinschaft (Grant 397230547).}

\begin{document}

\maketitle

\begin{center}
	\emph{Dedicated to Matthias Hieber on the occasion of his 60th birthday and in recognition of his brilliant work pointing to the future}
\end{center}

\begin{abstract}
Linear evolution equations are considered usually for the time variable being defined on an interval where typically initial conditions or time-periodicity of solutions are required to single out certain solutions. 
Here we would like to make a point of allowing time to be defined on a metric graph or network where on the branching points coupling conditions are imposed such that time can have ramifications and even loops. This not only generalizes the classical setting and allows for more freedom in the modeling of coupled and interacting systems of evolution equations, but it also provides a unified framework for initial value and time-periodic problems.
For these time-graph Cauchy problems questions of well-posedness and regularity of solutions for parabolic problems are studied along with the question of which time-graph Cauchy problems cannot be reduced to an iteratively solvable sequence of Cauchy problems on intervals. 
Based on two different approaches -- an application of the Kalton--Weis theorem on the sum of closed operators and an explicit computation of a Green's function -- we present the main well-posedness and regularity results. We further study some qualitative properties of solutions. While
 we mainly focus on parabolic problems
we also explain how other Cauchy problems can be studied along the same lines.  
 This is exemplified by discussing  coupled systems with constraints that are non-local in time akin to periodicity.
\end{abstract}

%\footnote{\DM{Der Artikel erscheint in einer Sammlung für Matthias Hieber. Hältst Du für angebracht, ihm den Artikel explizit zu widmen?}}

\section{Introduction}

Time has classically been considered as a linear phenomenon, especially in western cultures. This has been clearly mirrored in the physical description of the world, all the way from ancient Greek philosophy to modern partial differential equations of mathematical physics. 
Many real world phenomena can be -- more or less naively -- modeled as
abstract Cauchy problems
\begin{align}\label{eq:linear-inho}
\partial_t \psi(t) - A\psi(t)=f(t),
\end{align}
such as the heat, transport or Schrödinger equation, which are classically considered with domain for the time variable $t$ in a finite interval $[0,a]$ or a half-line $[0,\infty)$, and there cannot be a unique solution until an initial condition
$\psi(0)=g$ is imposed. Here, for simplicity one may have in mind a sectorial operator $-A$ in a Hilbert space $X$.

The western philosophy has, ever since Aristotle~\cite{Coo05} and perhaps Heraclitus, most commonly regarded time as a \textit{linear} instance that allows to order events according to the notions of  \textit{before} and \textit{after}: similar ideas are also typical in western monotheistic religions. It is folklore that, throughout the world, different cultures have had diverging approaches to the interpretation of time: some religions of Indian origin -- most notably Hinduism and Jainism, unlike Buddhism~\cite{Abe97,Cow99} -- postulate that the time consists of ages featuring repeating patterns, leading to a cyclic existence described by \textit{Kālacakra}; but also the cosmological implications of the \textit{Xiuhmolpilli} (52-year-cycles in Aztec calendar) or the \textit{Bak'tun} (144,000-day-cycles in Maya calendar) suggest a cyclic understanding of time~\cite{Mar12}, with such cycles conveniently clocking existence.

This does not necessarily lead to mathematical clashes: indeed, if the time variable $t$ is cyclic and hence lives in a torus $\mathbb{S}^1$ or the full real line $\R$, then looking for solutions of~\eqref{eq:linear-inho} amounts to inquire existence of periodic solutions. 

In each case the time domain is an oriented one-dimensional manifold, thus there is a clear direction at each point in time and a well-defined time before and after it. Going beyond this, there are different perceptions of time expressed for instance in the multiverse interpretation of quantum mechanics or in the discussions on closed time-like curves in general relativity. 
More recently, the theoretical physicist Carlo Rovelli has been advocating the necessity of giving up even the weakly ordered structure offered by Albert Einstein's conception of time. He writes in~\cite[Chapter 6]{Rov18}:

\begin{quote}
\emph{None of the pieces that time has lost (singularity, direction,
independence, the present, continuity) puts into question the fact that the	world is a network of events. On the one hand, there was time, with its many determinations; on the other, the simple fact that nothing is: things happen.}

\emph{The absence of the quantity ``time'' in the fundamental equations does not imply a world that is frozen and immobile. On the contrary, it implies a world in which change is ubiquitous, without being ordered by Father Time; without innumerable events being necessarily distributed in good order, or along the single Newtonian time line, or according to Einstein's elegant geometry.}
% The events of the world do not form an orderly queue, like the English. They crowd around chaotically, like Italians.}
\end{quote}

In this article, we would like to invite the reader to participate in a thought experiment and to assume time not to consist of a one-dimensional manifold, but rather of a \textit{metric graph} or \textit{network}. Such ramified structures consist -- roughly speaking -- of intervals glued together at their endpoints and allow for more freedom in the modeling of evolutionary systems in real and some possibly hypothetical applications. The purpose of this note is to widen the scope of classical evolution equations and to show how graphs can be used to model time evolution. The main idea and recurrent motive is to consider initial conditions as boundary conditions in time: we will make this more precise in the following.

We notice in passing that there do exist classical settings where the notion of one-dimensional time is generalized: In the context of analytic semigroups time is allowed to be in a sector of the complex plane as sketched in Figure~\ref{Fig1}.(d). This has a plethora of pleasant mathematical consequences, but it is not evident how to make sense of it physically.
Instead, we reckon that allowing time to live on network-like structures may have a practical interpretation as will be discussed in terms of examples.

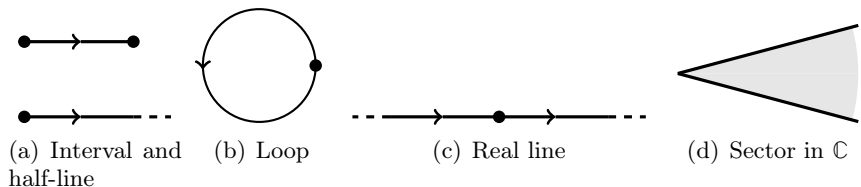
\begin{figure}[h]
	\begin{center}
		\subfigure[Interval and half-line]{
			\begin{tikzpicture}[scale=0.5]
			\fill[black] (0,0) circle (1ex);
			\draw[->, black, very thick] (0,0) -- (1.5,0);
			\draw[black, very thick] (1.5,0) -- (2.9,0);
			\fill[black] (2.9,0) circle (1ex);
			\fill[black] (0,-2) circle (1ex);
			\draw[->, black, very thick] (0,-2) -- (1.5,-2);
			\draw[black, very thick] (1.5,-2) -- (2.9,-2);
			\draw[black, very thick, dashed] (2.9,-2) -- (3.9,-2);
			\end{tikzpicture}
		}
		\subfigure[Loop]{
			\begin{tikzpicture}[scale=0.5]
			\draw[thick] (0,0) arc (0:360:1.5);
			\fill[black] (0,0) circle (1ex);
			\draw[->, black, very thick] (-3,0) -- (-3,-0.1);
			\end{tikzpicture}
		}	
		\subfigure[Real line]{
			\begin{tikzpicture}[scale=0.5]
			\fill[black] (0,-2) circle (1ex);
			\draw[->, black, very thick] (0,-2) -- (1.5,-2);
			\draw[black, very thick] (1.5,-2) -- (2.9,-2);
			\draw[black, very thick, dashed] (2.9,-2) -- (3.9,-2);
			\draw[black, very thick] (-1.5,-2) -- (0,-2);
			\draw[->, black, very thick] (-2.9,-2) -- (-1.5,-2);
			\draw[black, very thick, dashed] (-2.9,-2) -- (-3.9,-2);
			\end{tikzpicture}	
		}
				\subfigure[Sector in $\C$]{
				\begin{tikzpicture}[scale=0.8]
				\fill[fill=gray!20!white]
				(0,0) -- (3,0) arc (360:345:3);
				\fill[fill=gray!20!white]
				(0,0) -- (3,0) arc (0:15:3);
				\draw[-, black, very thick] (0,0) -- (3,-0.8); 
			\draw[-, black, very thick] (0,0) -- (3,0.8); 
				\end{tikzpicture} 	
			}
\caption{Classical time domains for evolution equations}\label{Fig1}
	\end{center}
\end{figure}

\subsection{From initial conditions to boundary conditions in time}
To begin with, considering the classical cases illustrated in Figure~\ref{Fig1}.(a)--\ref{Fig1}.(c) one 
first notices that for the real line or the torus there are no initial conditions, and in fact adding initial conditions would over-determine the system. For a bounded interval or the half-line, the initial value problem can be decomposed using linearity into two separate problems
\begin{align}\label{eq:icbc}
\begin{split}
\partial_t \psi_f(t) - A\psi_f(t)&=f(t), \quad \psi(0)=0, \quad \hbox{ and } \\ \partial_t \psi_0(t) - A\psi_0(t)&=0, \qquad \, \psi(0)=g.
\end{split}
\end{align}
Both equations can be analyzed in terms of semigroup theory: If $A$ generates a $C_0$-semigroup then the mild solutions to these equations are given by the variation of constants formula and the semigroup, i.e.,
\begin{align*}
\psi_f(t)=\int_0^t e^{A(t-s)}f(s) ds \quad \hbox{ and } \quad \psi_0(t)=e^{tA}g,
\end{align*}
where the solution to the inhomogeneous initial value problem
is $\psi=\psi_f +\psi_0$ and the solution space depends on the regularity of the data.

The problem on an interval $(0,a)$ with periodicity conditions exhibits similarities with the first equation in \eqref{eq:icbc}, and it can be written as 
\begin{align}\label{eq:icbcper}
\partial_t \psi(t) - A\psi(t)=f(t), \quad \psi(0)-\psi(a)=0.
\end{align}
This already indicates which possible `initial conditions' -- or rather `inhomogeneous boundary conditions in time' -- can be imposed, namely one can solve
\begin{align}\label{eq:icbcper_jump}
\partial_t \psi(t) - A\psi(t)=0, \quad \psi(0)-\psi(a)=g.
\end{align}
This means there is no freedom left for initial conditions, but one is free to choose any \textit{fixed jump condition} $\psi(0)-\psi(a)=g$, and the solution can be expressed (provided $\mathds{1}- e^{a \, A}$ is  invertible) as
\begin{equation}\label{eq:per0}
\psi_0(t)
= e^{tA}(\mathds{1}- e^{a\, A})^{-1}g \quad \hbox{for }t \in (0,a), 
\end{equation}
which solves \eqref{eq:icbcper_jump} on $(0,a)$. This solution can be extended to the full real line, it then solves
\begin{align*}
	\partial_t \psi(t) - A\psi(t)=0, \quad \psi(na)-\psi((n+1)a)= e^{naA} g \quad \hbox{for } n \in \Z.
\end{align*}
In particular, this extension does not lift to a solution on the torus. So, in order to interpret time as a loop, one has to consider the periodic extension of \eqref{eq:per0}. This is in general a non-continuous periodic function on $\R$ which then lifts to a function on the torus.

 The regularity of $\psi_0$ given in \eqref{eq:per0} clearly depends on the regularity of $(\mathds{1}- e^{a\, A})^{-1}g$ and therefore on $g$, as well as on the mapping properties of $(\mathds{1}- e^{a\, A})^{-1}$. The usual notions of mild, strong and classical solutions defined edge-wise  in the same way as e.g. in~\cite{EngNag00,DenHiePru03} can be naturally extended to the setting of metric graphs in time, see also Subsection~\ref{subsec:mild}.

Considering only the first equation in \eqref{eq:icbc}, this can be solved -- instead of using the variation of constants formula -- by means of operator theory by finding realizations of $\partial_t$ with initial condition $\psi(0)=0$ such that the sum of closed operators $\partial_t -A$ is invertible. For $L^p$-spaces in time this approach succeeded where the essential ingredient is the theorem of Kalton and Weis on the sum of closed operators. 
Similarly, equation \eqref{eq:icbcper} can be solved by considering a periodic realization of the time derivative.

\subsection{Time-graph Cauchy problem}

%The key observation is that 
Consider again evolution equations whose time domain are intervals, like in Figure~\ref{Fig1}.(a)--\ref{Fig1}.(c): both under initial and periodicity
conditions they can be split into a part with force, but homogeneous boundary condition in time, and a part without force and inhomogeneous boundary condition in time.
We therefore consider finitely many inhomogeneous evolution equations
\begin{align*}
(\partial_t - A_i) \psi_i = f_i \quad \hbox{on} \quad (0,a_i)\quad \hbox{for each} \quad i=1, \ldots, n,
\end{align*}
on time intervals of length $a_i>0$, $i=1,\ldots,n$, where we assume that $A_i$ are generators of analytic semigroups in Hilbert spaces $X_i$, $f_i\in L^2(0,a_i;X_i)$ are given, and the coupling is defined by 
\begin{align}\label{eq:introB}
\begin{bmatrix}
\psi_1(0) \\ \vdots \\ \psi_n(0)
\end{bmatrix} - \B \begin{bmatrix}
\psi_1(a_1) \\ \vdots \\ \psi_n(a_n)
\end{bmatrix}= \begin{bmatrix}
g_1\\ \vdots \\ g_n
\end{bmatrix},
\end{align}
where $\B$ is a bounded operator in $X_1\oplus \ldots \oplus X_n$ which encodes the geometry of the graph by means of transmission conditions, and $g_i\in X_i$ are given `inhomogeneous boundary conditions in time' in analogy to the fixed jump conditions for the periodic case.
This class of  \textit{time-graph  Cauchy problems} comprises the classical settings, where the classical initial value problem corresponds to $\B=0$, 
and the time-periodic problem is given by $\B=\mathds{1}$ with $g_i=0$ for $i=1, \ldots, n$.  

We present two strategies to solve this problem:
First, when all $g_i=0$, one can apply the Kalton--Weis theorem on the sums of closed operators for suitable time and space operators. 
Second, going beyond this, explicit formulae in terms of semigroups and transmission conditions as in \eqref{eq:per0} can be derived by a Green's function \textit{Ansatz} interpreting the system $\partial_t - A_i$ as a system of vector valued ordinary differential equation in time where inhomogeneous boundary conditions in time are included.

\subsection{First examples, results and outlook}
As a next step towards more non-standard examples, one can extend the time-periodic situation: %. One may want to extend the scope in order to look for ``spiral-like'' solutions: 
instead of pure periodicity, we may for instance impose a phase shift after one time-period $a>0$, i.e.,
%\footnote{\DM{Wir hatten ursprünglich hier den Text:\\
%\begin{align*}
%\psi(x,t+a) = \alpha \psi(x,t) \quad \hbox{for some} \quad \alpha\in \C,
%\end{align*} 
%which corresponds to $\B=\alpha\cdot \mathds{1}$ and $g=0$.
%Or having two or possibly even more phase shifts
%\begin{align*}
%\psi(x,t+2na) = \alpha_1 \psi(x,t), \quad \psi(x,t+(2n+1)a) = \alpha_2 \psi(x,t), \quad \quad \alpha_1,\alpha_2\in \C, \quad n\in \Z,
%\end{align*} 
%\\
%Ich muss zugeben, ich verstehe selber nicht mehr, wofür $x$ steht und vermute, dass es sich um ein Typo handelt. Habe deshalb $x$ im Folgenden die Abhängigkeit von $x$ gelöscht.}}
\begin{align*}
	\partial_t \psi(t) - A\psi(t)=f(t), \quad \psi(0) = \alpha \psi(a) \quad \hbox{for some} \quad \alpha\in \C,
\end{align*} 
which corresponds to $\B=\alpha\cdot \mathds{1}$ and $g=0$.
%Or having two or possibly even more phase shifts
%\begin{align*}
%\psi(t+2na) = \alpha_1 \psi(x,t), \quad \psi(t+(2n+1)a) = \alpha_2 \psi(t), \quad \quad \alpha_1,\alpha_2\in \C, \quad n\in \Z,
%\end{align*} 
As we will see later in Section~\ref{sec:Green} the solution to this problem is given -- provided that $A$ generates an analytic semigroup and the operator $1-\alpha e^{aA}$ is invertible -- by 
\begin{align}\label{eq:phaseshift}
	\psi(t) = \int_0^t e^{(t-s)A} f(s) ds + e^{tA}\int_0^a (1-\alpha e^{aA})^{-1} \alpha e^{(a-s)A} f(s) ds.  
\end{align}
In particular, extending this to the real line, one has
\begin{align*}
	\psi(na) - \alpha \psi((n+1)a)=
	 \int_0^{na} e^{(na-s)A} f(s) ds -\alpha \int_a^{(n+1)a}   e^{((n+1)a-s)A} f(s) ds, \quad n\in \Z,  
\end{align*}
and the
phase shift occurs only after the first  time-period starting at $0$ and ending at $a$, i.e., for $n=0$, as sketched in Figure~\ref{fig:loops}~(b). If instead one considers $\psi$ as a function on $[0,a)$ and extends it periodically to $\R$ by setting $\psi(t+an)=\psi(t)$ for $n\in \Z$, then for $\alpha\neq 1$ this  $\psi$ is a discontinuous periodic function with the additional property that
\begin{align*}
\psi(na)-\alpha \lim_{\varepsilon\to 0-}\psi((n+1)(a+\varepsilon))=0	\quad \hbox{for } n\in \Z,
\end{align*}
and this can be represented by   Figure~\ref{fig:loops}~(a). Another model is sketched in Figure~\ref{fig:loops}~(c). Here,  time is represented by the real line, a phase shift with phase $\alpha_1$ occurs at time $a_1$, and a second phase shift with phase $\alpha_2$ occurs at time $a_2$. 
%, and this corresponds to $\B=\diag(\alpha_1,\alpha_2)$ and $g=0$. 
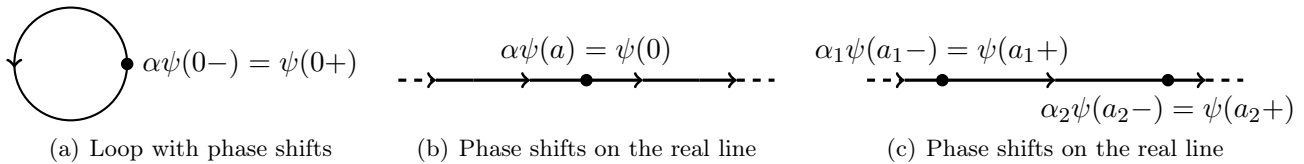
\begin{figure}[h]
	\begin{center}
		\subfigure[Loop with phase shifts]{	
			\begin{tikzpicture}[scale=0.5]
			\draw[thick] (0,0) arc (0:360:1.5);
			\fill[black] (0,0) circle (1ex);
			\draw[->, black, very thick] (-3,0) -- (-3,-0.1);
			\draw (0,0) node[right=2pt] {$\alpha \psi(0-)=\psi(0+)$};
			\end{tikzpicture}
		}	
		\subfigure[Phase shifts on the real line]{
		\begin{tikzpicture}[baseline = {(0,-.55)}, scale=0.5]
		%	\fill[black] (0,0) circle (1ex);
			\fill[black] (3,0) circle (1ex);
			\draw[->, black, very thick] (0,0) -- (1.5,0);
			\draw[black, very thick] (1.5,0) -- (3,0);
			\draw[->, black, very thick] (3,0) -- (4.5,0);
			\draw[ black, very thick] (4.5,0) -- (6,0);
			\draw (3,0) node[above=2pt] {$\alpha\psi(a)=\psi(0)$};	
		%	\fill[black] (6,0) circle (1ex);
				\draw[->, black, very thick] (6,0) -- (7,0);
			\draw[dashed, black, very thick] (7,0) -- (8,0);
			\draw[black, very thick] (-1,0) -- (0,0);
			\draw[dashed, ->,black, very thick] (-2,0) -- (-1,0);
					\end{tikzpicture}	
	}
	\subfigure[Phase shifts on the real line]{
		\begin{tikzpicture}[baseline = {(0,-.55)}, scale=0.5]
	\fill[black] (0,0) circle (1ex);
\draw[->, black, very thick] (0,0) -- (3,0);
\draw[ black, very thick] (3,0) -- (6,0);
\draw (0,0) node[above=2pt] {$\alpha_1\psi(a_1-)=\psi(a_1+)$};	
\fill[black] (6,0) circle (1ex);
\draw (6,0) node[below=2pt] {$\alpha_2\psi(a_2-)=\psi(a_2+)$};	
	\draw[->, black, very thick] (6,0) -- (7,0);
	\draw[dashed, black, very thick] (7,0) -- (8,0);
	\draw[black, very thick] (-1,0) -- (0,0);
\draw[dashed, ->,black, very thick] (-2,0) -- (-1,0);
\end{tikzpicture}	
}

%		\subfigure{	
%			\begin{tikzpicture}[scale=0.5]
%			\draw[thick] (0,0) arc (0:360:1.5);
%			\fill[black] (0,0) circle (1ex);
%			\fill[black] (-3,0) circle (1ex);
%			\draw[->, black, very thick] (-1.5,-1.5) -- (-1.6,-1.5);
%			\draw[->, black, very thick] (-1.6,1.5) -- (-1.5,1.5);
%			\draw (0,0) node[right=2pt] {$\alpha_1\psi(0-)=\psi(0+)$};
%			\draw (-3,0) node[left=2pt] {$\alpha_2\psi(0-)=\psi(0+)$};
%			\end{tikzpicture}			
%		}
	\end{center}
	\caption{Phase shifts}
\label{fig:loops}	
	%\label{loopwithshift}
\end{figure}

To illustrate various features of time-graphs one can consider the graphs depicted in Figure~\ref{Fig2}. Building on the initial example of time-periodicity, one can take its state at a certain time as input to a new system. 
This would correspond to the tadpole-like graph in Figure~\ref{Fig2}.(a) with matching of the type
\begin{align*}
\psi_1(a_1) = \psi_1(0), \quad \psi_2(0)=\psi_1(a_1), \quad \hbox{i.e.,}\quad \B= \begin{bmatrix}
1 & 0 \\ 1 & 0 
\end{bmatrix}, \quad g=0,
\end{align*}
where $\psi_1$ lives on the loop and $\psi_2$ lives on the adjacent interval. 
\begin{figure}[h]%\label{fig:various_graphs}
	\begin{center}
		\subfigure[Tadpole-like graph]{	\begin{tikzpicture}[scale=0.5]
			\draw[thick] (0,0) arc (0:360:1.5);
			\fill[black] (0,0) circle (1ex);
			\draw[->, black, very thick] (-3,0) -- (-3,-0.1);
			\draw[->, black, very thick] (0,0) -- (2,0);
			\end{tikzpicture}
		}	
		\subfigure[Splitting of a system]{			
			\begin{tikzpicture}[scale=0.5]
			\fill[black] (0,0) circle (1ex);
			\draw[->, black, very thick] (0,0) -- (1,0);
			\draw[black, very thick] (1,0) -- (2,0);
			\fill[black] (2,0) circle (1ex);
			\draw[->, black, very thick] (2,0) -- (3,1);
			\draw[black, very thick] (3,1) -- (4,2);
			\draw[black, very thick] (3,-1) -- (4,-2);			
			\draw[->, black, very thick] (2,0) -- (3,-1);
			\end{tikzpicture}
		}
		\subfigure[Joining of two systems]{\begin{tikzpicture}[scale=0.5]
			\fill[black] (0,2) circle (1ex);
			\fill[black] (0,-2) circle (1ex);
			\draw[->, black, very thick] (0,2) -- (1,1);
			\draw[->, black, very thick] (0,-2) -- (1,-1);
			\draw[black, very thick] (1,1) -- (2,0);
			\draw[black, very thick] (1,-1) -- (2,0);
			\fill[black] (2,0) circle (1ex);		
			\draw[->, black, very thick] (2,0) -- (3,0);
			\draw[black, very thick] (3,0) -- (5,0);	
			\end{tikzpicture}
		}	
		\subfigure[Graph with cycle]{	\begin{tikzpicture}[scale=0.5]
			\fill[black] (0,0) circle (1ex);
			\draw[->, black, very thick] (0,0) -- (1,0);
			\draw[ black, very thick] (1,0) -- (2,0);
			\fill[black] (2,0) circle (1ex);
			\draw[thick] (2,0) arc (-180:0:1);
			\draw[->, black, very thick] (3,-1) -- (3.1,-1);
			\draw[thick] (2,0) arc (180:0:1);
			\draw[->, black, very thick] (3,1) -- (3.1,1);
			\fill[black] (4,0) circle (1ex);
			\draw[->, black, very thick] (4,0) -- (5,0);
			\draw[black, very thick] (5,0) -- (6,0);
			\end{tikzpicture}
		}
		\subfigure[Graph with several loops]{	\begin{tikzpicture}[scale=0.5]
			\draw[thick] (0,0) arc (-90:270:1.5);
			\draw[thick] (0,0) arc (-90:270:1);
			\draw[thick] (0,0) arc (-90:270:0.5);
			\fill[black] (0,0) circle (1ex);
			\draw[->, black, very thick] (0,3) -- (0.01,3);
			\draw[->, black, very thick] (0,2) -- (0.01,2);
			\draw[->, black, very thick] (0,1) -- (0.01,1);
			\draw[->, black, very thick] (0,0) -- (2,0);
			\fill[black] (3,0) circle (1ex);
			\fill[black] (-3,0) circle (1ex);
			\draw[black, very thick] (2,0) -- (3,0);
			\draw[->, black, very thick] (-3,0) -- (-2,0);
			\draw[black, very thick] (-2,0) -- (0,0);
			\end{tikzpicture}
		}
		\subfigure[Tree graph]{	\begin{tikzpicture}[scale=0.6]
			\fill[black] (0,0) circle (1ex);
			\draw[->, black, very thick] (0,0) -- (0.75,0);
			\draw[black, very thick] (0.5,0) -- (1,0);
			\fill[black] (1,0) circle (1ex);
			\draw[->, black, very thick] (1,0) -- (1.5,0.5);
			\draw[black, very thick] (1.5,0.5) -- (2,1);
			\fill[black] (2,1) circle (1ex);
			\draw[->, black, very thick] (1,0) -- (1.5,-0.5);
			\draw[black, very thick] (1.5,-0.5) -- (2,-1);
			\fill[black] (2,-1) circle (1ex);
			\draw[->, black, very thick] (2,1) -- (2.5,1.5);
			\draw[black, dashed] (2.5,1.5) -- (2.9,1.9);
			\draw[->, black, very thick] (2,1) -- (2.5,0.5);
			\draw[black, dashed] (2.5,0.5) -- (2.9,0.1);
			\draw[->, black, very thick] (2,-1) -- (2.5,-1.5);
			\draw[black, dashed] (2.5,-1.5) -- (2.9,-1.9);
			\draw[->, black, very thick] (2,-1) -- (2.5,-0.5);
			\draw[black, dashed] (2.5,-0.5) -- (2.9,-0.1);
			\end{tikzpicture}
		}			
			\subfigure[Time-travel graph]{	\begin{tikzpicture}[scale=0.6]
				\fill[black] (0,0) circle (1ex);
				\draw[->, black, very thick] (0,0) -- (1,0);
				\draw[black, very thick] (1,0) -- (2,0);
				\fill[black] (2,0) circle (1ex);
				\draw[->, black, very thick] (2,0) -- (3,0);
				\draw[black, very thick] (3,0) -- (4,0);
				\fill[black] (4,0) circle (1ex);
				\draw[->, black, very thick] (4,0) -- (5,0);
				\draw[black, dashed] (5,0) -- (6,0);
				\draw[->, black, very thick] (3.1,-1) -- (3,-1);
				\draw[thick] (4,0) arc (0:-180:1);
			%	\draw[->, black, very thick] (2,0) -- (3,1);
			%	\draw[black, dashed] (3,1) -- (4,2);	
				\end{tikzpicture}
			}
		\subfigure[Time-travel-multiverse graph]{	\begin{tikzpicture}[scale=0.6]
			\fill[black] (0,0) circle (1ex);
			\draw[->, black, very thick] (0,0) -- (1,0);
			\draw[black, very thick] (1,0) -- (2,0);
			\fill[black] (2,0) circle (1ex);
			\draw[->, black, very thick] (2,0) -- (3,0);
			\draw[black, very thick] (3,0) -- (4,0);
			\fill[black] (4,0) circle (1ex);
			\draw[->, black, very thick] (4,0) -- (5,0);
			\draw[black, dashed] (5,0) -- (6,0);
			\draw[->, black, very thick] (3.1,-1) -- (3,-1);
			\draw[thick] (4,0) arc (0:-180:1);
			\draw[->, black, very thick] (2,0) -- (3,1);
			\draw[black, dashed] (3,1) -- (4,2);	
			\end{tikzpicture}
		}
		\subfigure[Matching different dynamics]{	\begin{tikzpicture}[scale=0.7]
			\fill[black] (0,0) circle (1ex);
			\draw[->, black, very thick] (0,0) -- (1,0);
			\draw (1,0) node[above] {$\partial_t -A_1$};
			\draw[black, very thick] (1,0) -- (2,0);
			\fill[black] (2,0) circle (1ex);
			\draw[->, black, very thick] (2,0) -- (3,0);
			\draw (3,0) node[below] {$\partial_t -A_2$};
			\draw[black, very thick] (3,0) -- (4,0);
			\fill[black] (4,0) circle (1ex);
			\draw[->, black, very thick] (4,0) -- (5,0);
			\draw (5,0) node[above] {$\partial_t -A_3$};
			\draw[black, dashed] (5,0) -- (6,0);
			\end{tikzpicture}
		}
\subfigure[First and second order Cauchy problems]{
			\begin{tikzpicture}[baseline = {(0,-.55)}, scale=0.6]
			\fill[black] (0,0) circle (1ex);
			\draw[->, black, very thick] (0,0) -- (1,0);
			\draw[ black, very thick] (1,0) -- (2,0);
			\draw (1,0) node[above=2pt] {$\partial_t-A$};
			\fill[black] (2,0) circle (1ex);
			\draw[->, black, very thick] (2,0) -- (3,0);
			\draw[ black, very thick] (3,0) -- (4,0);
			\fill[black] (2,0) circle (1ex);
			\draw (3,0) node[above=2pt] {$\partial_{t}^2-A$};	
			\fill[black] (4,0) circle (1ex);	
			\end{tikzpicture}
}
		\caption{Evolution equations on graphs}\label{Fig2}
	\end{center}
\end{figure}
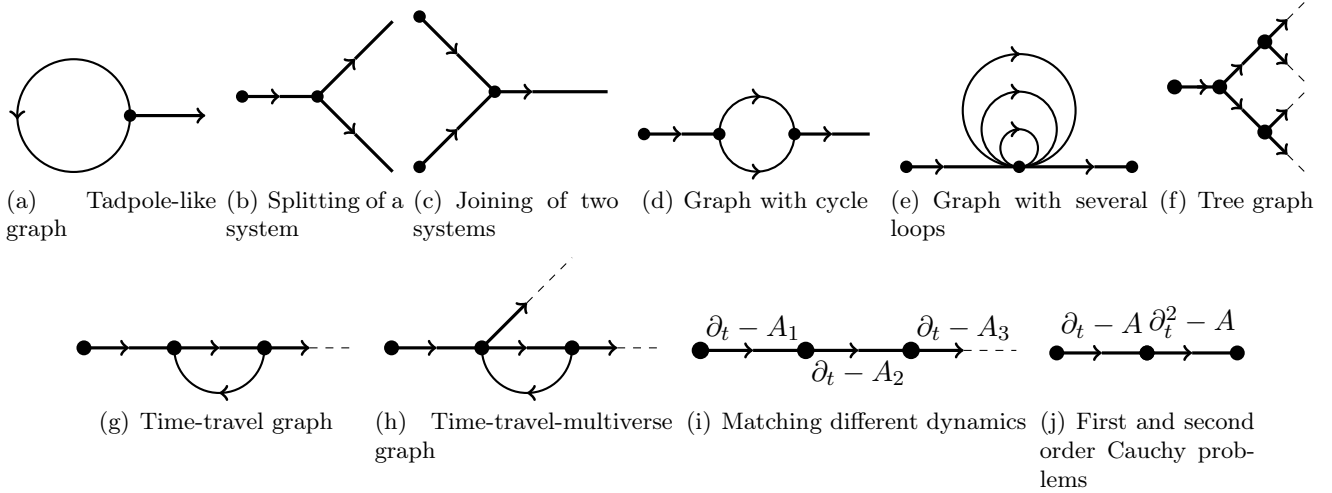

More generally, basic building blocks are the joining and the splitting of two systems -- as depicted in Figure~\ref{Fig2}.(b)-\ref{Fig2}.(c) -- which can be used to describe a system which splits into two non-interacting dynamics or two systems which interact after some time by means of some superposition. 
These blocks can be assembled to form graphs with cycle, see Figure~\ref{Fig2}.(d). Similarly, one may think of the interaction of various periodic systems with dynamics on the time-line, see Figure~\ref{Fig2}.(e) which shares some features with both Figure~\ref{Fig2}.(a) and Figure~\ref{Fig2}.(d).

Time-graph Cauchy problems can be understood as a system of Cauchy problems on intervals with possibly non-local constraints such as periodicity, fixed jump conditions, or certain symmetries.
Since the map $\B=(\B_{ij})_{1\leq i,j\leq n}$ is a block operator matrix with $\B_{ij}\colon X_j \rightarrow X_i$, one can rewrite \eqref{eq:introB} as 
\begin{align*}
\psi_j(0)- \B_{jj}\psi_j(a_j)=g_j +\sum_{i\neq j} \B_{ij}\psi_i(a_i), \quad j=1,\ldots,n,
\end{align*}
that is, a Cauchy problem is assigned on each interval and their `jump conditions' are interdependent. If $\B_{jj}\neq 0$ the Cauchy problem on $(0,a_j)$ is non-local and resembles periodicity, and for $\B_{jj}=0$ the Cauchy problem on $(0,a_j)$ is an initial value problem.

Time-graphs with oriented loops can also be used to model closed loops and other control theoretical gadgets, cf.~\cite{JacZwa12}.
One can think also of signals that after a certain time are processed differently as illustrated in Figure~\ref{Fig2}.(i).
This means that a system changes its character after a certain time. For instance a heat equation is followed after a certain time by a transport process that after a certain time turns again into a heat equation: 
thus modeling time delays in a diffusive process.
Moreover, couplings at the vertices of a time-graph can be frequency dependent, and thus frequency dependent dynamics can be modeled, too. 
Also, there are some more non-standard situations where time-graphs come into play. A tree graph as depicted in Figure~\ref{Fig2}.(f) can serve as an illustration for the multiverse interpretation of quantum mechanics, where it is assumed that, in contrast to a probabilistic interpretation, each possible state is actually attained, but each in one separated universe. Figure~\ref{Fig2}.(g) and \ref{Fig2}.(h) give some possibilities how one may represent time travel -- independent of its actual physical possibility -- using time-graphs, see also Section~\ref{sec:timetravel} below.

Our main result states the well-posedness of such time-graph models, under some compatibility assumption on the matrix $\B$, which encodes the transmission conditions in time, and the `spatial' operators operators $A_i$. In particular, a generalized variation of constants formula is obtained, allowing us to derive additional mapping properties.

The question of whether the time-graph Cauchy problem reduces to a sequence of Cauchy problems on intervals which can be solved iteratively is traced back to the block structure of $\B$, and it is pointed out that loops which are reflected by the transmission conditions $\B$ prevent such iterative solvability and therefore in such situations one indeed needs tools for global solvability such as for the case of periodicity. The methods developed for the case of parabolic problems can be adapted also for some non-parabolic problems such as Schrödinger equations, wave equations, or even coupled dynamics of different types as first and second order Cauchy problems as illustrated in Figure~\ref{Fig2}~(j).

\subsection{Organization of the paper}
In the subsequent Section~\ref{sec:classicalee} we recapitulate key elements of the classical theory of evolution equations, some of which are necessary in order to develop our approach to time-graphs. Thereafter, in Section~\ref{sec:metricgraphs}, the notion of networks and function spaces thereon are made precise. In Section~\ref{sec:der} the Banach space-valued time derivative operator on graphs with couplings and the spatial operator are studied. In Section~\ref{sec:DoreVenni} the time-graph problem for the case $g=0$ is tackled, using the Kalton--Weis sum theorem on commuting operators applied to the time derivative and the spatial operator, where some compatibility assumptions on the boundary conditions are required. 
Section~\ref{sec:Green} follows a more direct approach computing the Green's function for the time-graph problem explicitly. This gives our main result on the solvability of the time-graph Cauchy problem for $g$ in a trace space under less restrictive compatibility conditions. 
Section~\ref{sec:sym} addresses the question 
under which condition solutions to time-graph problems can be reduced to Cauchy problems on intervals.
In Sections~\ref{sec:examples}  we discuss a few examples, focusing on specific instances of time graphs and broaching extensions to classes of non-parabolic evolution equations,  including Schrödinger, wave and mixed-order equations.

Some of the suggested settings may look mostly motivated by science-fictional or hypothetical physical scenarios, as they may allow for loss of causality:  In Section~\ref{sec:timetravel} we discuss these and further related aspects by commenting on tentative interpretations of evolution supported on network-type time structures.

\section{Classical Cauchy problems}\label{sec:classicalee}
Many of the methods applied here make use of classical results on evolution equation theory and initial value problems. 
It is well-established that the initial value problem 
\begin{align}\label{eq:ACP}
\partial_t \psi(t) - A\psi(t)=f(t), \quad \psi(0)=g
\end{align}
with $A$ being a closed linear operator on a Banach space $X$ has for all $g\in X$ a unique mild solution if and only if $A$ generates a $C_0$-semigroup on $X$, cf.\ \cite[Thm.~3.1.12]{AreBatHie10}, where at least
$f\in L^1(0,a;X)$ is admissible,  
$a>0$.
If $X$ is a Hilbert space and $g=0$, the stronger condition of \textit{maximal $L^2$-regularity} amounts to requiring that there is for all $f\in L^2(0,a;X)$ a unique solution $\psi$ of~\eqref{eq:ACP} in the maximal $L^2$-regularity space, i.e.,
\begin{align*}
%MR(0,a;A;X):= 
\psi\in L^2(0,a;D(A))\cap \{\psi\in W^{1,2}(0,a;X)\colon \psi(0)=0\},
\end{align*} 
such that 
\[
\norm{\psi}_{L^2(0,a;D(A))} + \norm{\psi}_{W^{1,2}(0,a;X)} \leq C \norm{f}_{L^{2}(0,a;X)}
\]
for a constant $C>0$ independent of $f$. Maximal $L^2$-regularity holds if and only if the semigroup generated by $A$ on $X$ is analytic. This is related to the notion of sectorial operators: considering sectors in the complex plane
\begin{align*}
\Sigma_\omega := \{\lambda\in  \C\setminus \{0\}\colon |\hbox{arg}(\lambda)|<\omega\}, \quad \omega\in (0,\pi),
\end{align*}
 recall that a closed densely defined linear operator $B$ is \textit{sectorial of angle $\omega\in (0,\pi)$} if 
 \begin{itemize}
 \item $\sigma(B)\subset\overline{\Sigma_\omega}$ and 
 \item $\sup \{ \norm{ \lambda(\lambda-B)^{-1}}\colon \lambda \in \C\setminus \{0\}\colon\nu\leq |\hbox{arg}(\lambda)|\leq \pi \}<\infty$ for all $\nu\in (\theta, \pi)$,
 \end{itemize}
 cf.\ \cite[Theorem 1.11 ff.]{KunWei04}. 
Note that if $B=-A$ is sectorial of angle smaller than $\pi/2$, than $A$ is the generator of an analytic semigroup. In the literature, there are several, slightly diverging  definitions of sectorial operators. For example, in~\cite[Definition 4.1]{EngNag00} it is the generator itself (rather than its negative) which is called ``sectorial''; while in~\cite[Chapter 3, \S 3.10]{Kat80} an operator on a Hilbert spaces is called ``sectorial'' if its numerical range lie in a sector.

For Banach spaces $X$ of class UMD, maximal $L^p$-regularity can be characterized using the notions of $\cR$-sectoriality and $\cH^{\infty}$-calculus, where one implication follows from the Dore--Venni-type sum theorem of Kalton and Weis on commuting operators~\cite[Thm.~6.3]{KalWei01}, cited here in Theorem~\ref{thm:DoreVenni} below. 
The key idea in the original Dore--Venni Theorem and its generalizations is to look at evolution equations  on a Banach space $X$ as stationary equations on a Bochner space of $X$-valued functions. 

\begin{theorem}[Sum theorem of Kalton and Weis]\label{thm:DoreVenni}
Suppose that $A\in \cH^{\infty}(X)$ and $B\in \cR\cS(X)$ are commuting operators such that $\phi_A^{\infty}+\phi_B^R<\pi$. Then $A+B$ is closed with domain $D(A+B)=D(A)\cap D(B)$, $A+B\in \cR\cS(X)$ with $\phi_{A+B}\leq \max\{\phi_A^{\infty},\phi_B^R\}$, and for some constant $C>0$
	\begin{align*}
	\norm{Ax}_X +\norm{Bx}_X \leq C\norm{(A+B)x}_X, \quad x\in D(A)\cap D(B).
	\end{align*} 
	The operator $A+B$ is invertible if $A$ or $B$ is invertible.
\end{theorem}
In the following we will seldom use this result in its full generality, as we mostly restrict to the case of Hilbert spaces; we refer the interested reader to the classic monograph~\cite{DenHiePru03} by Denk, Hieber and Pr\"u\ss  where all these notions are introduced.
Theorem~\ref{thm:DoreVenni} is formulated for a Banach space $X$.  If $X$ is a Hilbert space, however, then the notions of $\cR$-sectoriality and sectoriality agree. 
We recall that whenever $-A$ is sectorial, the solution $\psi(t):=e^{tA}g$ lies in $D(A)$ for  all $t>0$ and all initial data $g\in X$; and that moreover $\psi$ lies for all $p\in (1,\infty)$ in the maximal $L^p$-regularity space whenever the initial data belong to the \textit{trace space}, i.e., 
$g\in (X,D(A))_{1-1/p,p}$, 
given by the real interpolation functor 
$(\cdot,\cdot)_{\theta,p}$, cf.\ \cite[\S~3.4]{PruessSimonett}.

The \textit{Ansatz} using the Kalton and Weis sum theorem has been applied successfully by Arendt and Bu in~\cite{AreBu02,AreBu04,AreBu04b}. 
In particular, 
the fact that both time domains $\R$ and $\mathbb S^1$ are groups has allowed them to apply methods of harmonic analysis and to deliver a comprehensive theory of Cauchy problems with \textit{time-periodic} boundary conditions. 
A general scheme for  periodic and almost periodic solutions to semilinear equations has been proposed by Hieber and co-authors, cf.\ \cite{Hieber2016_per, Hieber2017_per} and also \cite{Hieber2019_per, Hieber2017b_per, HieberStinner, HieberGaldi}, where in particular for applications in fluid mechanics semigroup theory plays an important role, cf. \cite{Hieber2020}. 
For a similar approach where the stationary part is treated separately,
%\footnote{\DM{Hier schrieben wir:
%\begin{quote}
%where the stationary part, i.e., $k=0$, is treated separately [...]
%\end{quote}
%Der/die Referee fragt sich, was $k$ sei, und ich gebe zu, ich selber kann das nicht beantworten. Daher habe ich das ``$k=0$'' gelöscht.}}
and in particular applications to quasi- and semi-linear problems see also the works of Kyed and co-authors, cf.\ \cite{KyeSau17, EitKye17, CelKye18}.  
 Existence of time-periodic solutions for (linear or even nonlinear) hyperbolic equations is well-known for a large class of problems, cf.\ the comprehensive monograph~\cite{Vej82}.

\section{Finite metric graphs}\label{sec:metricgraphs}

\subsection{Finite graphs}\label{sec:finitegraphs}
A \emph{graph} is a $4$-tuple 
$$\Ge = \left( \verts, \Ie,\Ee, \partial \right),$$ 
where $\verts$ denotes the set of \textit{vertices}, 
$\Ie$ the set of \textit{internal edges} 
and $\Ee$ the set of \textit{external edges}, with $\Ee\cap \Ie=\emptyset$. We refer to elements of the set $\Ee \cup \Ie$ collectively as \textit{edges}. To avoid notational ambiguities, we also assume $\verts\cap\Ee=\verts\cap \Ie=\emptyset$.
In order to fix an orientation, one distinguishes \textit{incoming} $\Ee_-$ and \textit{outgoing} $\Ee_+$ external edges, where $\Ee=\Ee_-\cup \Ee_+$ and $\Ee_-\cap \Ee_+=\emptyset$. 

The structure of the graph is given by the \textit{boundary map} $\partial$. On one hand, it assigns to each internal edge $i\in \Ie$ 
an ordered pair of vertices $\partial (i)=\left(\partial_-(i),\partial_+(i)\right)\in \verts \times \verts$, 
where $\partial_-(i)$ is called its \textit{initial vertex} 
and $\partial_+(i)$ its \textit{terminal vertex}. 
On the other hand, each incoming external edge $e_-\in \Ee_-$ and each outgoing external edge $e_+\in \Ee_+$ is associated by means of $\partial(e_-)=\partial_-(e_-)$ and $\partial(e_+)=\partial_+(e_+)$ with a single vertex (its initial and terminal vertex, respectively). A graph is called \textit{balanced} if $|\Ee_-|=|\Ee_+|$. 
We will see that orientations do play a role only when we study evolution equations that are of first (or, more generally, odd) order in time; in the case of even time order equations, orientations are only imposed for the sake of a consistent parametrization. %}
A graph is called \textit{finite} if $\abs{\verts}+\abs{\Ie}+\abs{\Ee}<\infty$ 
	and a finite graph is called \textit{compact} if $\Ee=\emptyset$.

The structure of the network is given by the $|\verts|\times |\Ee\cup \Ie|$-\emph{outgoing} and \emph{ingoing incidence matrices}
$I^+:=(\iota^+_{\mv \me})$ and $I^-:=(\iota^-_{\mv \me})$ defined by
\begin{equation}\label{inout}
	\iota^+_{\mv \me}:=\left\{
	\begin{array}{rl}
		1, & \hbox{if } \partial_-(\me)=\mv,\\
		0, & \hbox{otherwise},
	\end{array}
	\right.
	\qquad\hbox{and}\qquad
	\iota^-_{\mv \me}:=\left\{
	\begin{array}{rl}
		1, & \hbox{if } \partial_+(\me)=\mv,\\
		0, & \hbox{otherwise.}
	\end{array}
	\right.
\end{equation}
This encodes the structure of the graph and allows one to define directions on $\Ge$. The network $\Ge$ is the directed graph whose signed incidence matrix is $I:=(\iota_{\mv \me})$, defined by $I:=I^+-I^-$; we will occasionally need the \textit{underlying undirected graph}, which is fully defined by the (signless) incidence matrix $I:=I^+ +I^-$. Roughly speaking, a \textit{directed graph} 
is the version of the graph where one can move only along the prescribed direction, while for the \textit{undirected graph} $\Ge$ one can move into both directions.

\subsection{Function spaces on metric graphs}
A graph $\Ge$ is endowed with the following metric structure. 
Each internal edge $i\in \Ie$ is associated 
with an interval $[0,a_i]$, with $a_i>0$, 
such that its initial vertex corresponds to~$0$ 
and its terminal vertex to~$a_i$. 
Each external edge $e\in \Ee_-$ and $e\in \Ee_+$ is associated to a half line $[0,\infty)$ and $(-\infty,0]$, respectively,
such that $\partial(e)$ corresponds to~$0$. 
The numbers~$a_i$ are called \textit{lengths} of the internal edges $i\in \Ie$ and they are collected into the vector 
\begin{align*}
\au=\{a_i\}_{i\in \Ie}\in (0,\infty)^{\abs{\Ie}}.
\end{align*} 
The couple consisting of a finite graph endowed with a metric structure is called a \textit{metric graph} $(\Ge,\au)$. 
The metric on the undirected metric graph $(\Ge,\au)$ is defined via minimal path lengths along connected vertices, while for the directed metric graph minimal path lengths along connected vertices is computed taking into account the directions.

Let, for each $j\in \Ie \cup \Ee$, $X_j$ be  a complex Banach space with norm $\norm{\cdot}_{X_j}$. Then any collection of functions 
\begin{align*}
\psi_j\colon I_j \rightarrow X_j, \quad j\in \Ie \cup \Ee \hbox{ with } I_j= \begin{cases} (0,a_j), & \mbox{if} \ j\in \Ie, \\ 
(0,\infty), &\mbox{if} \ j\in \Ee_-, \\ 
(-\infty,0), &\mbox{if} \ j\in \Ee_+,
\end{cases} 
\end{align*}
can be identified with a map
\begin{align}\label{eq:mapsongraphs}
\psi\colon \bigsqcup_{j\in \Ie \cup \Ee } I_j \rightarrow \bigsqcup_{j\in \Ie \cup \Ee } X_j \quad \hbox{with} \quad \psi(t)=\psi_j(t) \quad \hbox{for } t\in I_j,
\end{align}
where the notation for elements in
\begin{align*}
t_j=(t,j)\in\bigsqcup_{j\in \Ie \cup \Ee } I_j \quad  \hbox{and} \quad  \psi_j=(\psi,j)\in \bigsqcup_{j\in \Ie \cup \Ee } X_j
\end{align*}
 is 
shortened to $t$ and $\psi$, and
occasionally we write 
slightly redundantly $\psi_j(t)=\psi_j(t_j)$. 
The metric graph $(\Ge,\au)$ is identified with a quotient of $\bigsqcup_{j\in \Ie \cup \Ee } \overline{I_j}$, and therefore $t\in(\Ge,\au)$ is identified with $t=t_j\in \overline{I_j}$ for some $j\in \Ee\cup \Ie$.
Similarly,  the maps $\psi$ defined as in \eqref{eq:mapsongraphs} can be identified with maps on $(\Ge,\au)$,
 where on the vertices  in general a set of values can be attained,   because for the different edges adjacent to a vertex the edgewise defined functions $\psi_j$ can in general take different values.

Equipping each edge of the oriented or non-oriented metric graph with
the one-dimensional vector-valued Bochner--Lebesgue measure, 
one obtains a measure space. One defines 
\begin{equation*} 
\int_{\Ge} \psi := \sum_{j\in \Ie \cup \Ee} \int_{I_j} \psi(t_j) \, dt_j %+ \sum_{e_-\in \Ee_-} \int_{0}^{\infty} \psi(x_{e_-})\, dx_{e_-}, 
%+ \sum_{e_+\in \Ee_+} \int_{-\infty}^{0} \psi(x_{e_+})\, dx_{e_+}, 
\end{equation*}
where $dt_{j}$ 
refers to integration with respect to the Bochner--Lebesgue measure on $I_j$.  
We set
\begin{align*}
\cX:= \bigoplus_{j\in\Ie \cup \Ee} X_j,
\end{align*}
and introduce, with a slight abuse of notation, several related spaces:
For $p\in (1,\infty)$ the space 
\begin{align*}
L^p(\Ge,\au;\cX):= \bigoplus_{j\in\Ie \cup \Ee} L^p(I_j;X_j) %\oplus \left(\bigoplus_{e\in\Ee_-} L^p(-\infty,0;X)\right) \oplus \left(\bigoplus_{e\in\Ee_+} L^p(0,\infty;X_j)\right)
\end{align*}
defines a Banach space, and indeed a Hilbert space provided $p=2$ and  $X_j$ are Hilbert spaces; the canonical norm and inner product are given by
\begin{align*}
\norm{\psi}_{L^p}:=
\left(\sum_{j\in \Ie \cup \Ee} \int_{I_j} \norm{\psi(t_j)}_{X_j}^p   \, dt_j\right)^{1/p}
 \quad \hbox{and}
  \quad \langle \psi, \varphi\rangle_{L^2}= \sum_{j\in \Ie \cup \Ee} \int_{I_j} \langle \psi_j(t_j), \varphi_j(t_j) \rangle_{X_j} \, dt_j,
\end{align*}
respectively.
The corresponding Sobolev spaces are defined for $p\in [1,\infty)$ and $m \in \N$ by
\begin{align*}
W^{m,p}(\Ge,\au;\cX):= \bigoplus_{j\in\Ie\cup \Ee} W^{m,p}(I_j;X_j). %\oplus \left(\bigoplus_{e\in\Ee_-} %W^{m,p}(-\infty,0;X)\right) \oplus \left(\bigoplus_{e\in\Ee_+} W^{m,p}(0,\infty;X)\right) .
\end{align*}
Recall that for $\psi\in W^{m,p}(\Ge,\au;\cX)$, $m\in \N$, $p\in [1,\infty)$ traces up to the order $m-1$ are well-defined, i.e., $$\psi^{(n)}(\partial_{\pm}(j)) \in X_j, \quad j\in \Ie\cup \Ee_{\pm}, \hbox{ where } 0\leq n\leq m-1.$$
Also, using $$W_0^{m,p}(I_j;X_j)= \{\psi_j\in W^{m,p}(I_j;X_j) \colon \psi_j^{(n)}\vert_{\partial I_j}=0, \quad 0\leq n\leq m-1 \}$$
one sets 
\begin{align*}
W^{m,p}_0(\Ge,\au;\cX):=\bigoplus_{j\in\Ie\cup \Ee} W_0^{m,p}(I_j;X_j).
\end{align*}

\section{Operators on metric graphs}\label{sec:der}

As a first step to study the motivating problem, i.e.,
\[
\partial_t \psi(t) -A(t)\psi(t)=f(t), \qquad t\in (\Ge,\au),
\]
the derivative operator with transmission conditions on graphs is analyzed.

\subsection{Derivative operators on graphs}
One considers the $n$-th derivative operators $D_n$ on graphs formally given by 
\begin{align*}
(D_n\psi)_j = \partial_t^n \psi_j, \quad j\in \Ie\cup \Ee,
\end{align*}
where one can define minimal and maximal operators in $L^{p}(\Ge,\au;\cX)$ by
\begin{align*}
D(D_n^{\min}):=W_0^{n,p}(\Ge,\au;\cX) \subset D(D_n^{\max}):=W^{n,p}(\Ge,\au;\cX).
\end{align*}
These are closed linear operators. If $p=2$ and each $X_j$ is a Hilbert space, one has
%\begin{align*}
$(D_n^{\min})^* = (-1)^{n}D_n^{\max}$,
%\end{align*}
and hence $D_n^{\min}$ is symmetric if $n$ is even and skew-symmetric if $n$ is odd. In this article, the focus lies on the first and second derivative operator, for which we use the notation
\begin{align*}
D_t:=D_1 \quad \hbox{and} \quad D_{tt}:=D_2.
\end{align*}

\subsection{Accretive coupling conditions for the first derivative}
When considering the first derivative operator, it is assumed that $\Ge$ is balanced, i.e., there are as many outgoing as incoming external edges.
From now on, let $X_j$ be Hilbert spaces. On $L^2(\Ge,\au;\cX)$ a class of $m$-accretive realizations of $D_t$ defined by boundary conditions is studied,
%\footnote{\DM{Der/die Referee fragt, was wir mit ``is presented'' meinen, und ich muss zugeben, ich kann den Ausdruck auch nicht verstehen.}}, 
i.e., we consider operators $D_t^{b.c}$ with 
$$
D_t^{\min} \subset D_t^{b.c} \subset D_t^{\max},$$
where $\rho(D_t^{b.c})\neq \emptyset$, and
\begin{align*}
\Real \langle D_t^{b.c}\psi,\psi \rangle \geq 0, \quad \hbox{for all } \psi\in D(D_t^{b.c}).
\end{align*} 
Integrating by parts yields the following Lagrange identity for the first derivative operator 
\begin{equation}\label{intbp}
\int_{\Ge}\langle \psi',\varphi\rangle_X + \int_{\Ge}\langle \psi, \varphi' \rangle_X 
= [ \psi,\varphi]_{\partial\Ge, D_t}, \quad \psi,\varphi\in D(D_t^{\max}),
\end{equation}
where 
\begin{align*}
	\left[\psi,\varphi\right]_{\partial \Ge, D_t}:=
	\sum_{e\in \Ie\cup \Ee_+}\langle\psi(\partial_+(e)),\varphi(\partial_+(e)
\rangle_X-	
	\sum_{e\in \Ie\cup \Ee_-}\langle\psi(\partial_-(e)),\varphi(\partial_-(e)	\rangle_X.
\end{align*}
One introduces the space of boundary values
\begin{align*}
\Ke:=
\bigoplus_{i\in \Ie} X_i \oplus \bigoplus_{e\in \Ee_-} X_e
 \simeq 
 \bigoplus_{i\in \Ie} X_i \oplus \bigoplus_{e\in \Ee_+} X_e,
\end{align*}
where the claimed isomorphism holds because the graph is balanced, i.e., $|\Ee_+|=|\Ee_-|$: the vectors of boundary values $\underline{\psi}_-\in \Ke$ and $\underline{\psi}_+\in \Ke$ are then defined  by
\begin{align}\label{eq:RandwerteD1_1}
\underline{\psi}_+:=
\begin{bmatrix}\{ \psi(\partial_+(i))\}_{i\in \Ie} \\
\{ \psi(\partial_+(e))\}_{e\in \Ee_p}
\end{bmatrix}, \quad 
\underline{\psi}_-:=\begin{bmatrix}\{\psi(\partial_-(i))\}_{i\in \Ie} \\
\{\psi(\partial_-(e))\}_{e\in \Ee_p}
\end{bmatrix}
\quad \hbox{and} \quad [\psi]:= \begin{bmatrix}\underline{\psi}_+ \\ \underline{\psi}_-\end{bmatrix}\in \Ke^2,
\end{align}
where for a fixed bijection 
\begin{align*}
p\colon \Ee_+\rightarrow \Ee_- \quad \hbox{one sets}\quad \Ee_p=\{(e_+,p(e_+))\in \Ee_+\times \Ee_-\colon e_+\in \Ee_+\},
\end{align*}
i.e., one orders the outgoing and incoming edges into pairs, and defines
\begin{align*}
\partial_+(e):=\partial_+(e_+), \quad \partial_-(e):=\partial_-(p(e_+)), \hbox{ where } e=(e_+,p(e_+))\in \Ee_p.
\end{align*}
Hence, one obtains
\begin{align}\label{eq:boundaryterms}
\left[\psi,\varphi\right]_{\partial \Ge, D_t}= \langle\underline{\psi}_+,\underline{\varphi}_+ \rangle_{\Ke}
-\langle\underline{\psi}_-,\underline{\varphi}_- \rangle_{\Ke} = \langle \underline{\psi}, J \underline{\varphi}\rangle_{\Ke^2}, \quad \hbox{where } J=\begin{bmatrix}
\mathds{1}_{\Ke} & 0 \\ 0 & -\mathds{1}_{\Ke} 
\end{bmatrix}.
\end{align}
For any  subspace $\Me\subset \Ke^2$ one can define a realization by
\begin{align*}
D_t(\Me) \psi :=D_t^{\max}\psi = \psi', \quad
D(D_t(\Me)):=\left\{ \psi\in D(D_t^{\max})\colon \underline{\psi}\in \Me \right\}
\end{align*}
with the extremal cases $D_t(\Ke^2)=D_t^{\max}$ and $D_t(\{0\})=D_t^{\min}$ 
being clearly edge-wise decoupled, and couplings can be implemented by means of boundary conditions.
\begin{lemma}\label{lem:closedDt}
The operator $D_t(\Me)$ is closed if and only if $\Me\subset \Ke^2$ is closed.
\end{lemma}
\begin{proof}
	If $\Me\subset \Ke^2$ is closed, then $\psi_n\to \psi$ and $\psi_n'\to \varphi$ in $L^2(\Ge,\au;\cX)$ for $\psi_n\in D(D_t(\Me))$ imply first due to the closedness of $D_t^{\max}$ that $\psi\in D(D_t^{\max})$ and $\varphi=\psi'$. Second, due to the boundedness of the trace operator one has in $\Ke^2$ that $\underline{\psi_n}\to\underline{\psi}\in\Me$.
	
	If $\Me\subset \Ke^2$ is not closed, then there exist a Cauchy sequence $(\underline{\psi_n})\subset \Me$ with $\underline{\psi_n}\to\underline{\psi}\notin\Me$. Note that there exist smooth cut-off functions $\eta_{j}^{\pm}\colon I_j \rightarrow [0,1]$ with $\eta_{j}^{\pm}=1$ close to $\partial_{\pm}(e_j)$ and zero around $\partial_{\mp}(e_j)$. Then $\psi_{n,j}=\eta_j^{+} (\underline{\psi}_{+})_j +\eta_j^{-} (\underline{\psi}_{-})_j$ defines  functions $\psi_{n}\in D(D_t(\Me))$ with $\psi_{n}\to \psi$ and $\psi_{n}'\to \psi'$, but $\psi\notin D(D_t(\Me))$.
\end{proof}
Here, the following type of boundary conditions is considered.
Let $\B\in \mathcal{L}(\Ke)$  be a bounded operator on $\Ke$.
Note that $\B$ is a block operator matrix given with respect to the decomposition of $\Ke:=
\bigoplus_{i\in \Ie\cup \Ee_-} X_i$,
 i.e.,
\begin{align*}
\B = (\B_{ij})_{i,j\in \Ie \cup \Ee_-} \quad \hbox{with} \quad \B_{ij}\in \cL(X_j,X_i).
\end{align*}
For such $\B\in \mathcal{L}(\Ke)$ we consider the boundary conditions defined by 
\begin{equation}\label{eq:bcD1}
\underline{\psi}_- = \B\underline{\psi}_+.
\end{equation}
One defines the operator 
\begin{align*}
D_t(\B):=D_t(\Me(\B)), \quad \hbox{where } \Me(\B) := 
\{[\psi]\in \Ke^2\colon \underline{\psi}_- - \B\underline{\psi}_+=0\}.
\end{align*}
Under additional assumptions these boundary conditions force the numerical range of $D_t(\B)$ and $D_t(\B)^*$ to lie in a left half-plain of the complex plain.

 \begin{lemma}[Adjoint operator and numerical range]\label{lem:adjoint}
Let $\B\in \mathcal{L}(\Ke)$. Then $D_t(\B)$ is closed, its Hilbert space adjoint in $L^2(\Ge,\au;\cX)$ is given by	
\begin{align*}
D_t(\B)^*\varphi = -\varphi', \quad
 	D(D_t(\B)^*)=\{\varphi \in D(D_t^{\max})\colon \B^*\underline{\varphi}_- - \underline{\varphi}_+=0 \},
\end{align*}
 and furthermore
 	\begin{align*}
 	\Real\langle D_t(\B)\ \psi,\psi \rangle&= 
 	\tfrac{1}{2}\langle (\mathds{1} -\B^*\B)\underline{\psi}_+, \underline{\psi}_+\rangle_{\Ke}, \quad \psi \in D(D_t(\B)),\\
 	\Real\langle D_t(\B)^*\varphi,\varphi \rangle&= 
 	\tfrac{1}{2}\langle (\mathds{1}-\B\B^*)\underline{\varphi}_-, \underline{\varphi}_-\rangle_{\Ke}, \quad \varphi \in D(D_t(\B)^*).
 	\end{align*}
 \end{lemma}
\begin{proof}
By Lemma~\ref{lem:closedDt} $D_t(\B)$ is closed since $\Me(\B)$ is closed.
Note that from $D_t^{\min}\subset D_t(\B)\subset D^{\max}_1$ it follows by taking adjoints that $-D_t^{\min}\subset D_t(\B)^*\subset -D_t^{\max}$. Hence it follows from \eqref{eq:boundaryterms} that
\begin{align*}
D(D_t(\B)^*)=\{\varphi\in D_t^{\max}\colon J\underline{\varphi}\in \Me(\B)^{\perp}\}.
\end{align*}
Note that 
\begin{align*}
\Me(\Be)=\ker \begin{bmatrix}
\mathds{1} & -\B
\end{bmatrix} \perp \Ran \begin{bmatrix}
\mathds{1} \\ -\B^*
\end{bmatrix} = 
\ker \begin{bmatrix}
\B^* & \mathds{1}
\end{bmatrix},
\quad \hbox{hence }  J (\Me(\B)^{\perp}) =\ker \begin{bmatrix}
\B^* & -\mathds{1}
\end{bmatrix}.
\end{align*}

Moreover, for $\psi\in D(D_t(\B))$ one obtains by integration by parts
\begin{align*}
 	\Real\langle D_t(\B)\psi,\psi \rangle&=
 	\tfrac{1}{2}
 \left( \langle D_t(\B)\psi,\psi \rangle + \overline{\langle D_t(\B)\psi,\psi \rangle} \right)\\
& =	\tfrac{1}{2} \left( \langle \psi',\psi \rangle + \langle \psi,\psi' \rangle \right) 
 = \tfrac{1}{2}\left(\langle \underline{\psi}_+, \underline{\psi}_+\rangle_{\Ke} - \langle \underline{\psi}_-,\underline{\psi}_-\rangle_{\Ke} \right)\\
 &= \tfrac{1}{2} \left(\langle \underline{\psi}_+, \underline{\psi}_+\rangle_{\Ke} - \langle \B\underline{\psi}_+,\B\underline{\psi}_+\rangle_{\Ke} \right)
 = \tfrac{1}{2} \langle (\mathds{1}-\B^*\B)\underline{\psi}_+, \underline{\psi}_+\rangle_{\Ke}.
 	\end{align*}
A similar proof yields the claimed identity for $\Real\langle D_t(\B)^*\psi,\psi \rangle$.
\end{proof}

\begin{remark}[Spectral inclusion]\label{rem:spectrum}
	Note that if $\B$ is a contraction then $\sigma(D_t(\B))\subset \{z\in \C\colon \Real z \leq 0\}$ since the spectrum is contained in the closure of the numerical range. The case $\sigma(D_t(\B))=\emptyset$ can occur, and in particular for a compact graph with $\B=0$ (which corresponds to the boundary condition $\psi(\partial_-(i))=0$ for all $i\in \Ie$) one has $\sigma(D_t(\B))=\emptyset$.
\end{remark} 

\begin{proposition}[M-accretivity and invertibility of $D_t(\B)$]\label{propo:maccretive}
Let $\B\in \mathcal{L}(\Ke)$. 
\begin{enumerate}
	\item[(a)] If $\B$ is a contraction on $\Ke$, i.e., $\norm{\B}_{\mathcal{L}(\Ke)}\leq 1$, then $D_t(\B)$ is m-accretive in $L^2(\Ge,\au;\cX)$;
	\item[(b)] If $\B$ is a strict contraction, i.e., $\norm{\B}_{\mathcal{L}(\Ke)}< 1$, then $D_t(\B)$ is boundedly invertible,
\item[(c)] if $\B$ is unitary, i.e., $\B^*\B=\B\B^*=\mathds{1}$, then $D_t(\B)$ is skew-self-adjoint, i.e., $D_t(\B)^*=-D_t(\B)$.
\end{enumerate}
\end{proposition}
\begin{proof}
	It is a direct consequence of Lemma~\ref{lem:adjoint} that if $\norm{\B}_{\mathcal{L}(\Ke)}\leq 1$, then
	\begin{align*}
	\Real\langle D_t(\B)\psi,\psi \rangle \geq 0 \quad \hbox{and} \quad \Real\langle D_t(\B)^*\varphi,\varphi \rangle\geq 0
	\end{align*}
	for all $\psi\in D(D_t(\B))$ and all $\varphi\in D(D_t(\B)^*)$, respectively, i.e., $D_t(\B)$ is m-accretive~\cite[Cor.~3.17]{EngNag00}. Recall that one has for the operator norm in $\Ke$ that 
	$\norm{\B}_{\mathcal{L}(\Ke)}^2 =\norm{\B^*\B}_{\mathcal{L}(\Ke)}=\norm{\B\B^*}_{\mathcal{L}(\Ke)}$. Hence, one obtains from Lemma~\ref{lem:adjoint} the 
	following proposition, where part (b) follows using Remark~\ref{rem:spectrum}.
\end{proof}

By well-established results about operators with bounded $H^{\infty}$-calculus in Hilbert spaces, cf.\ \cite[5.2.2.~Thm.]{Are04} and also \cite[Chapt.~11]{KunWei04}, \cite[Cor.~7.1.8]{Haa06}, the following holds;  the notation $\phi^{\infty}$ for the angle of bounded $H^\infty$-calculus is introduced in~\cite[\S~4.5]{Are04}.
\begin{corollary}[Bounded $H^{\infty}$-calculus for $D_t(\B)$]\label{cor:Hinfty}
If 
$\B\in \cL(\Ke)$ is a contraction, then $D_t(\B)$ has a bounded $H^{\infty}\big(L^2(\Ge,\au;\cX)\big)$-calculus of angle $\phi_{D_t(\B)}^{\infty}=\frac{\pi}{2}$.
\end{corollary}

\subsection{Spatial operators}
As before we assume that $X_j$ are Hilbert spaces. For each edge $j\in \Ie\cup \Ee$ let
$A_j$ be a given operator in $X_j$ with $D(A_j)\subset X_j$.
We consider the abstract time-graph-Cauchy problem
\begin{equation}\label{eq:firstCP}
\left\{
\begin{split}
(\partial_t - A_j)\psi_j(t_j) &=f_j(t_j),\qquad t_j\in\bigsqcup_{j\in \Ie \cup \Ee } I_j,\\
 \underline{\psi}_- -\B \underline{\psi}_+ &=0.
\end{split}
\right.
\end{equation}

Note that the operators $A_j$ in $X_j$ induce operators in $L^2(I_j;X_j)$ which with a slight abuse of notation are also denoted by $A_j$ and $D(A_j)=L^2(I_j;D(A_j))$. Using this we define
the operator $A_\Ee$ in $L^2(\Ge,\au;\cX)$: it acts on functions supported on the time branches by
\begin{align}\label{def:A}
D(A_\Ee):=\bigoplus_{j\in \Ie\cup \Ee} L^2(I_j;D(A_j)),\quad
(A_\Ee\psi)_j := A_j\psi_j, 
\end{align}
and with this the Cauchy problem \eqref{eq:firstCP} can be formulated as a maximal regularity problem
\begin{align}\label{eq:firstCP_A}
(D_t(\B) - A_\Ee)\psi =f, \quad\hbox{where } \psi \in D(D_t(\B))\cap D(A_\Ee) \hbox{ for } f\in L^2(\Ge,\au;\cX).
\end{align}
Moreover, the operators $A_j$ induce an operator $A_{\cV}$ in the space of boundary values $\Ke$ that acts on functions supported on the vertices by
\begin{align}\label{def:Apm}
D(A_{\cV}):= \bigoplus_{j\in \Ie\cup \Ee} D(A_j), \quad
(A_{\cV}\, \underline{\psi})_j := A_j\underline{\psi}_j, 
\end{align}
which in turn induces an operator in
$\Ke^2$ by
\begin{align}\label{def:AKe}
D(A_{\cV^2}):= D(A_{\cV})\oplus D(A_{\cV}), \quad
A_{\cV^2}[\psi]  &:= \begin{bmatrix}A_{\cV}\underline{\psi}_- \\ A_{\cV}\underline{\psi}_+\end{bmatrix}, \quad [\psi]=\begin{bmatrix}\underline{\psi}_- \\ \underline{\psi}_+\end{bmatrix}.
\end{align}
The following lemma is straightforward.
\begin{lemma}[Spectrum of induced operators]
	Let $A_j$ be operators in $X_j$ with domain $D(A_j)$ for $j\in \Ie\cup \Ee$. Then for the induced operators $A_{\Ee}$ in $L^2(\Ge,\au;\cX)$, $A_{\cV}$ in $\Ke$, and $A_{\cV^2}$ in $\Ke^2$ the following holds:
	\begin{enumerate}
		\item[(a)] $\sigma(A_{\Ee})=\sigma(A_{\cV})= \sigma(A_{\cV^2})=\bigcup_{j\in \Ie\cup \Ee}\sigma(A_j)$ as an equality of sets, i.e., without counting multiplicities; 
		\item[(b)] If $-A_j$ are sectorial of angle $\phi_{-A_j}\in [0,\pi)$, then $-A_{\Ee},-A_{\cV},-A_{\cV^2}$ are sectorial with same sectoriality angle
		\begin{align*}
		\phi_{A_{\Ee}}=\phi_{A_{\cV}}=\phi_{A_{\cV^2}}=\max_{j\in \Ie\cup \Ee}\phi_{A_j}. 
		\end{align*}
	\end{enumerate}
\end{lemma}

\section{The Kalton and Weis sum theorem and the parabolic operator}\label{sec:DoreVenni}

\subsection{Solvability of the inhomogeneous problem with homogeneous boundary conditions}
Having specified time-derivative and spatial operators, one can now  define the parabolic operator
\begin{align*}
P(\B):=D_t(\B)-A_\Ee \quad \hbox{with} \quad P(\B)=D(D_t(\B))\cap D(A_\Ee).
\end{align*}
The Kalton--Weis sum theorem, formulated here in Theorem~\ref{thm:DoreVenni}, can now be applied to $D_t(\B)$ and $A_\Ee$ using Corollary~\ref{cor:Hinfty} and assuming that the $A_\Ee$ is sectorial and commuting  with $D_t(\B)$. This gives the well-posedness for the time-graph Cauchy problem with homogeneous initial conditions and inhomogeneous right hand-side.

\begin{proposition}\label{prop:invertible}
Let $\Ge$ be balanced and  $\B$ be a contraction in $\Ke$. Let  $X_j$ be for  all $j\in \Ie\cup \Ee$ Hilbert spaces and $-A_j$ sectorial operators of angle $\phi_{-A_i}<\pi/2$ on $X_j$. Assume that $D_t(\B)$ and $A_\Ee$ are resolvent commuting.
Then the operator $P(\B)$
is closed. 

If furthermore $A_\Ee$ or $D_t(\B)$ are boundedly invertible, then so is $D_t(\B)-A_\Ee$ and in this case there is a constant $C=C(\B,\Ge,\au)>0$ such that for any $f\in L^2(\Ge,\au;\cX)$ there is a unique solution $\psi$ to \eqref{eq:firstCP} with
	\begin{align*}
	\psi \in D_t(\B) \cap D(A_\Ee) \hbox{ and } \norm{\psi'}_{L^2(\Ge,\au;\cX)} + \norm{A_\Ee\psi}_{L^2(\Ge,\au;\cX)} \leq C \norm{f}_{L^2(\Ge,\au;\cX)}.
	\end{align*}
\end{proposition}
\begin{remark}
A criterion to assure that the operators $D_t(\B)$ and $A_\Ee$  commute is that 
$(A_{\cV}-\lambda)^{-1}$ and $\B$ commute for $\lambda\in \rho(A_{\cV})$.
\end{remark}

\subsection{Trace spaces and the parabolic operator}
The approach using the Kalton--Weis result on commuting operators allowed us to find a simple way how to check solvability for the time-graph Cauchy problem with homogeneous boundary data. However, the condition that $D_t(\B)$ and $A_{\Ee}$ commute seems too strict since \eqref{eq:firstCP} makes sense without it, and in fact closedness of the parabolic operator can be ensured under weaker assumptions. 

For notational simplicity  we assume from now on  that there are no external edges, i.e., the time-graph is assumed to be compact. 
Considering the maximal parabolic operator
\begin{align*}
P^{\max}:=D_t^{\max} - A_\Ee,
\end{align*}
where  $\Ee=\emptyset$, one defines the corresponding trace space
\begin{align*}
\Ke_A := [X,D(A_\cV)]_{1/2}= \bigoplus_{j\in \Ie} [X_j,D(A_j)]_{1/2},
\end{align*}
where $[\cdot,\cdot]_\theta$ for $\theta\in (0,1)$ denotes the complex interpolation functor. Recall that for sectorial $-A$ one has the continuous embedding
\begin{align}\label{eq:traceembedding}
D(D_t^{\max} -A_\Ee) \hookrightarrow \bigoplus_{j\in \Ie} BUC(I_j;[X_j,D(A_j)]_{1/2}),
\end{align}
where $BUC$ stands for the space of bounded uniformly continuous functions, cf.\ \cite[Section 3.4]{PruessSimonett} or \cite[Theorem 4.10.2]{Ama95}.

\begin{definition}[Boundary conditions compatible with trace space]\label{def:compatiblebc_trace}
	The operator $\B\in \cL(\Ke)$ is said to be {\textit{compatible with $\Ke_A$}} if it restricts to an operator in $\Ke_A$, i.e., $\B\vert_{\Ke_A}\in \cL(\Ke_A)$ holds.
\end{definition}
\begin{remark}\label{rem-s6}
	\begin{enumerate}
		\item[(a)] The actual definition of trace spaces $(\cX,D(A_\cV))_{1-1/p,p}$ uses the real interpolation functor  for $p\in (1,\infty)$, and here it is used that for $p=2$ one has $[\cdot,\cdot]_{1/2}=(\cdot,\cdot)_{1/2,2}$. If $D(A_\cV)\subset \cX$ is dense, then all interpolation spaces 
		$(\cX,D(A_\cV))_{\theta,p},[ \cX,D(A_\cV)]_{\theta}\subset \cX$ for $\theta\in [0,1], p\in (1,\infty)$
		are dense in $\cX$.		
		\item[(b)] 	Note that for $A_j$ having bounded imaginary powers and $A_j$ injective one has $[X_j,D(A_j)]_{1/2}=D(A^{1/2})$, and compatibility with the trace space $\Ke_A$ in the sense of Definition~\ref{def:compatiblebc_trace} holds provided one has
		that $A_{\cV}^{1/2}$ and $\B$ commute.
	\end{enumerate}
\end{remark} 

\begin{lemma}[Closedness of the parabolic operator]
	Let each $-A_j$ be sectorial of angle smaller than $\pi/2$,
	and let $\B\in \cL(\Ke)$ such that $\B\in \cL(\Ke)$ is compatible with $\Ke_A$. Then $P(\B)$
	is a closed operator on $L^2(\Ge,\au;\cX)$.
\end{lemma}

\begin{proof}
	One shows first that $P^{\max}$ is closed. Note that $P^{\max}$ decouples the edges and hence it is sufficient to prove closedness for a graph consisting of a single interval $[0,a]$. Consider the operator $P(\B)=P_{0,\delta}$ for $\B=0$ on $[-\delta,a_i]$ for $\delta>0$. This is closed and to trace back this property to $P^{\max}$ one considers continuous extension and restriction operators
	\begin{align*}
	E\colon D(P^{\max}) \rightarrow D(P_{0,\delta})\hbox{ and } R \colon D(P_{0,\delta}) \rightarrow D(P^{\max}) 
	\end{align*} 
	with $R\circ E = \mathds{1}_{D(P^{\max})}$, where the extension can be realized for instance by even reflection and then multiplying by a cut-off function with value one on $[-\delta/2,a]$ and zero in a neighborhood of $-\delta$. Then $P^{\max}=R \circ P_{0,\delta}\circ E$ and closedness can be proved straightforward.

	Now, let $\varphi_n\in D(P(\B))$ with 
	\begin{align*}
	\varphi_n \to \varphi \quad \hbox{and} \quad 	P(\B)\varphi_n \to \psi \quad \hbox{in } L^2(\Ge,\au;\cX).
	\end{align*}
	Then by closedness of $P^{\max}$ and since $P(\B)$ is a restriction of $P^{\max}$, $\varphi\in D(P^{\max})$ and $\psi=P^{\max}\varphi$. Using \eqref{eq:traceembedding}, it follows that
	$\underline{\varphi_n}_{\pm} \to \underline{\varphi}_{\pm}$, and hence $\varphi\in D(P(\B))$. 
\end{proof}

\section{The parabolic operator and the Green's functions approach}\label{sec:Green}
The operator theoretical consideration of the parabolic operator gives information on the solvability for homogeneous boundary data.
However, it does not provide a solution formula, and it does not include the case of in-homogeneous boundary data. 
To address these issues we supplement our findings by computing explicitly the Green's function for \eqref{eq:firstCP}.

\subsection{Green's function for the parabolic problem}\label{sec:Green-sub}
Now, we are in the position to collect suitable assumptions for the time-graph Cauchy problem; we stress that 
the following are more general than the ones in Proposition~\ref{prop:invertible}, where here $\Ee=\emptyset$ has been assumed for notational simplicity only. 
\begin{assumption}\label{ass:D1L2_2}
Let $\Ee=\emptyset$ and $\B\in \cL(\Ke)$. Let
 $X_j$ be a Hilbert space  and $-A_j$ a sectorial operator of angle $\phi_{-A_j}<\pi/2$ on $X_j$ for each $j\in \Ie$.  
\end{assumption}
In the following a solution formula is derived generalizing the variation of constants formula from semigroup theory. 
Note that square integral maps
\begin{align*}
k_{ij}\colon I_i \times I_j\rightarrow \cL(X_j;X_i), \quad i,j\in \Ie
\end{align*}
define integral operators acting on $L^2(\Ge,\au;\cX)$  via
\begin{align*}
\psi= \{\psi_i\}_{i\in \Ie} \mapsto  \left\{\sum_{j\in \Ie}\int_{I_j} k_{ij}(t_i,s_j) \psi_j(s_j) ds_j\right\}_{i\in \Ie} 
\end{align*}
In this sense, the Green's function for zero initial conditions, i.e., for $\B=0$, is
\begin{eqnarray}\label{eq:r0}
	\{r_0(t,s;A_\Ee)\}_{j,l}:= \begin{cases}e^{(t_j-s_j)A_j} & \hbox{if }j=l \hbox{ and } t_j\geq s_j, \\ 0  &\hbox{otherwise},
	\end{cases} \quad  t_j,s_l\in\bigsqcup_{j\in \Ie \cup \Ee } I_j.
\end{eqnarray} 	
Since each operator $-A_j$ is sectorial, it generates an  analytic $C_0$-semigroup; in particular, $e^{t_j A_j}$ is a well-defined bounded linear operator on $X_j$ for each $t_j$ in the time branch $I_j$. In the following we will adopt for $\tu = \{t_j\}_{j\in I_j}$, the notation
\begin{align*}
e^{\tu A}\colon \Ke\rightarrow \Ke, \quad \{\underline{\psi}_{j}\}_{j\in\Ie} \mapsto  \{e^{t_j A_j}\underline{\psi}_{j}\}_{j\in\Ie},
\end{align*}
and hence $e^{\tu A}\in \mathcal L(\Ke)$ is a diagonal block operator matrix in $\Ke$.

\begin{proposition}[Inhomogeneous problem with homogeneous boundary conditions]		
	\label{prop:resolvent}
Under the Assumption~\ref{ass:D1L2_2}, let $\B$ be compatible with $\Ke_A$, and let $(\mathds{1}-\B \, e^{\au A})\vert_{\Ke_A}$ be boundedly invertible in $\Ke_A$.
 Then $P(\B)$ is boundedly invertible, i.e., for each $f\in L^2(\Ge,\au;\cX)$ there exists a unique solution $\psi$ to \eqref{eq:firstCP} in $D(D_t(\B))\cap D(A_\Ee)$. This $\psi$
		 is given by
	\begin{align*}
	\psi=\int_{\Ge}r(\cdot,s;\B,A_\Ee)f(s) ds,
	\end{align*}
		where
	\begin{equation}\label{eq:ker-r}
r(t,s;\B,A_\Ee):=r_0(t,s;A_\Ee)+r_1(t,s;\B,A_\Ee)
	\end{equation}
	with $r_0(t,s;A_\Ee)$ given by \eqref{eq:r0} and
	\begin{align*}
	 r_1(t,s;\B,A_\Ee): = e^{\tu A} (\mathds{1}-\B \, e^{\au A})^{-1}\B e^{(\au-\su)A}.
	\end{align*} 
\end{proposition}

\begin{remark}\label{rem:invertibilityinKe}
	\begin{enumerate}[(a)]
		\item Note that $e^{\au A}\colon \Ke_A \rightarrow D(A_\cV)\subset \Ke_A$, and therefore if $\B$ is compatible with $\Ke_A$, then also $\mathds{1}-\B \, e^{\au A}$ is compatible with $\Ke_A$. 
		\item Moreover if $\mathds{1}-\B \, e^{\au A}\in \cL(\Ke)$ and $(\mathds{1}-\B \, e^{\au A})\vert_{\Ke_A}\in \cL(\Ke_A)$ hold, then $(\mathds{1}-\B \, e^{\au A})\vert_{\Ke_A}^{-1}\in \cL(\Ke_A)$ implies that 
		$(\mathds{1}-\B \, e^{\au A})^{-1}\in \cL(\Ke)$. To this end, knowing that $(\mathds{1}-\B \, e^{\au A})\vert_{\Ke_A}$ is closable in $\Ke$, it is sufficient  to prove that $(\mathds{1}-\B \, e^{\au A})\vert_{\Ke_A}^{-1}$ is closable in $\Ke$, cf.\ \cite[Lemma~2.28]{Git2012}. Now let
		\begin{align*}
		(\underline{\psi}_n)_{n\in \N}\subset \Ke_A \quad \hbox{with } \underline{\psi}_n \to 0 \hbox{ and } (\mathds{1}-\B \, e^{\au A})\vert_{\Ke_A}^{-1}\underline{\psi}_n \to \underline{\varphi} \hbox{ in } \Ke \hbox{ as } n\to \infty.
		\end{align*}
		Then since $(\mathds{1}-\B \, e^{\au A})\in \cL(\Ke)$
		  one has  
		 \begin{align*}
		 (\mathds{1}-\B \, e^{\au A})(\mathds{1}-\B \, e^{\au A})\vert_{\Ke_A}^{-1}\underline{\psi}_n=\psi_n \to (\mathds{1}-\B \, e^{\au A})\underline{\varphi}, 
		 \end{align*} 
		and since $\underline{\psi}_n \to 0$ the claim follows.
	\end{enumerate}
\end{remark}

\begin{proof}[Proof of Proposition~\ref{prop:resolvent}]
A solution to the equation \eqref{eq:firstCP} is on each edge $i \in \Ie$ of the form 
\begin{equation}\label{eq:boch}
\psi_i(t_i)= e^{A_it_i}c_i + \int_0^{t_i} e^{A_i(t_i-s_i)}f(s_i) d s_i, 
\end{equation}
for some vector $c_i\in X_i$ that is ``inherited'' from the final state in the preceding edges. Indeed, the boundary condition can be used to determine $c_i$. Since
\begin{align*}
\underline{\psi}_-= \{c_i\}_{i\in \Ie} \quad \hbox{and} \quad \underline{\psi}_+= \{e^{A_ia_i}c_i + \int_0^{a_i} e^{A_i(a_i-s_i)}f(s_i) d s_i\}_{i\in \Ie},
\end{align*}
recalling that $a_i$ denotes the length of the edge $i$,
the condition $\underline{\psi}_-=\B \underline{\psi}_+$ gives
\begin{align*}
c_{i} = \sum_{j=1}\B_{ij} ( e^{A_ja_j}c_j + \int_0^{a_j} e^{(a_j-s_j)A_j}f(s_j) d s_j).
\end{align*}
Hence we obtain the vector-valued identity for $c=\{c_i\}_{i\in \Ie}$
\begin{align*}
(\mathds{1}-\B \, e^{\au A})c= \B \{\int_0^{\au} e^{(\au-\su)A}f(s) d s\},
\end{align*}
and because $\mathds{1}-\B \, e^{\au A}$ is assumed to be invertible
\begin{align*}
c= (\mathds{1}-\B \, e^{\au A})^{-1}\B \{\int_0^{\au} e^{(\au-\su)A}f(s) d s\},
\end{align*}
whence
\[
\psi(t)= \int_0^{\au} e^{A\underline{t}}(\mathds{1}-\B \, e^{\au A})^{-1}\B e^{(\au-\su)A}f(s) ds + \int_0^{\underline{t}} e^{A(\tu-\su)}f(s) ds.
\]
Recall that the Green's function is the resolvent operator's integral kernel, i.e., a function $r(t,s):=r(t,s;\B,A_\Ee)$ such that 
it defines a left and right inverse of $P(\B)$, i.e.,
\begin{itemize}
\item[(a)] $\varphi(t)=(D_t(\B)-A_\Ee)\int_{\Ge} r(t,s)\varphi(s) ds$ for $\varphi\in L^2(\Ge,\au;\cX)$,
\item[(b)] $\psi(t)=\int_{\Ge}r(t,s) (D_t(\B)-A_\Ee)\psi(s) ds$ for $\psi\in D(P(\B))$.
\end{itemize}

	First, note that 
	\begin{equation}\label{eq:psi0psi1}
	\psi_{f}^0:= \int_{\Ge}r_0(\cdot,s;A)f(s) ds \quad \hbox{and}\quad \psi_{f}^1:= \int_{\Ge}r_{1}(\cdot,s;\B,A)f(s) ds
		\end{equation}
	solve $(\partial_t-A_j)(\psi_{f}^0)_j=f_j$ and $(\partial_t-A_j)(\psi_{f}^1)_j=0$ on each edge $j\in \Ie$, respectively, where one applies the classical variation of constants formula
	and the properties of the semigroups $e^{t_jA_j}$. Hence $\psi=\psi_f^0+\psi_f^1$ solves $(\partial_t-A_j)\psi_j=f_j$ on each edge with $\psi\in D(P^{\max})$. Here, $\psi_f^1$ is the correction term for the variation of constants term $\psi_f^0$ assuring that the boundary conditions are satisfied. 
		
Secondly, one has to prove that $\psi$ satisfies the boundary conditions, and indeed
\begin{align*}
\underline{\psi}_-&=\underline{\psi_1}_-= (\mathds{1}-\B \, e^{\au A})^{-1}\B \int_{\Ge} e^{A(\au-\su)} f(s) ds, \\
-\B\underline{\psi}_+&= -\B \int_{\Ge} e^{(\au-\su)A} f(s) ds - \B e^{\au A} (\mathds{1}-\B \, e^{\au A})^{-1}\B \int_{\Ge} e^{(\au-\su)A} f(s) ds,
\end{align*} 
hence $\underline{\psi}_- - \B\underline{\psi}_+=0$. We conclude that $\int_{\Ge} r(\cdot,s;\B,A)\cdot ds$ is the right inverse of $P(\B)$.

Because the adjoint kernels 
\begin{equation}\label{eq:rokern2}
	\{r_0(t,s;A_\Ee)^*\}_{j,l}:= \begin{cases}e^{(s_j-t_j)A_j^\ast} & \hbox{if }j=l \hbox{ and } t_j< s_j, \\ 0  &\hbox{otherwise},
	\end{cases} \quad  t_j,s_l\in\bigsqcup_{j\in \Ie \cup \Ee } I_j
\end{equation}
consist of the Green's function for the time-reversed problems
\begin{align*}
(-\partial_{t_j} -A_j^*) \psi_j =f_j, \quad \psi_j(a_j)=0,\quad j\in \Ie,
\end{align*} 
and since $ (-\partial_{t_j} -A_j^*) e^{(a_j-t_j)A^*} =0$ one has
\[
(-D_t -A^*)\int_{\Ge}r_1(s,t;\B,A)^*f(s) ds=0
\]
 with
\begin{align*}
r_1(s,t;\B,A)^* = e^{(\au-\tu)A^*} \B^*(\mathds{1}-e^{\au A^*} \, \B^*)^{-1} e^{\su A^*}
\end{align*}
and concerning the boundary conditions
\begin{align*}
\underline{\psi}_+&=\underline{\psi_1}_+= \B^*(\mathds{1}-e^{\au A^*} \, \B^*)^{-1} \int_{\Ge} e^{\su A^*} f(s) ds, \\
-\B^*\underline{\psi}_-&= -\B^* \int_{\Ge} e^{\su A^*} f(s) ds - \B^* e^{\au A^*} \B^*(\mathds{1}-e^{\au A^*} \, \B^*)^{-1} \int_{\Ge}e^{\su A^*} f(s) ds,
\end{align*} 
hence $\underline{\psi}_+ - \B^*\underline{\psi}_-=0$.

To conclude, note that $\mathds{1}- e^{A^*\au}\,\B^*$ is invertible if and only if so is its adjoint $\mathds{1}-\B \, e^{\au A}$. We have thus proven that the adjoint of $\int_{\Ge} r(\cdot,s;\B,A)\cdot ds$ is the right inverse of $P(\B)^*$: we hence take adjoints and find $\int_{\Ge} r(\cdot,s;\B,A)\cdot ds \, P(\B)= \mathds{1}_{D(P(\B))}$. We conclude that $\int_{\Ge}r(\cdot,s;\B,A)\ ds$ is also the left inverse of $P(\B)$.
\end{proof}

\begin{remark}\label{rem:conddelio}
Two sufficient conditions for invertibility of $\mathds{1}-\B e^{\au A}$ are that each $A_i$ are m-dissipative and $\B$ is a strict contraction; or else that each $A_i+\epsilon$ is m-dissipative for some $\epsilon>0$ and $\B$ is a contraction.
\end{remark}

For general $\B$ the resolvent of $D_t(\B)$ can be obtained applying Proposition~\ref{prop:resolvent} for $A_j:=\lambda \mathds{1}_{X_j}$,  $\lambda\in\C$ which induces an operator $A_{\Ee}=\lambda_{\Ee}$.
 \begin{corollary}[Resolvent of $D_t(\B)$]
 	\label{prop:resolvent1}
Let $\Ee=\emptyset$ and $\B\in \cL(\Ke)$. 
If $\mathds{1}-\B e^{\lambda\au}$ is  invertible in $\Ke$, then $\lambda\in \rho(D_t(\B))$ and
 	the unique solution $\psi\in D(D_t(\B))$ to 
 	\begin{align*}
 	(D_t(\B)-\lambda) \psi =f\quad \hbox{for} \quad f\in L^2(\Ge,\au;\cX)
 	\end{align*}
 	is given by $ 	\psi=\int_{\Ge}r(\cdot,s;\B,\lambda_{\Ee})f(s) ds$.
 \end{corollary}
The inverse of the parabolic operator $P(\B)$ can be seen as being given by a functional calculus where the spectral parameter in Corollary~\ref{prop:resolvent1} is replaced by the operator $A$. This is akin to the case of classical semigroups, where the solution operator of the ordinary differential equation $(\partial_t -\lambda)\psi=f, \quad \psi(0)=\psi_0,$
is considered, and semigroup theory -- interpreted as functional calculus for the exponential functions -- allows one to ``replace $\lambda$ by some generator $A$''.

\subsection{Inhomogeneous boundary conditions}
So far, we have implicitly  
focused on the case of 0-boundary conditions imposed on \textit{sources} of the time-graph, i.e., at the initial endpoints of those time branches that have no predecessors. This is clearly a relevant limitation and would e.g.\ lead to identically vanishing solutions as soon as $f\equiv 0$.
Initial conditions can be introduced by interpreting them as inhomogeneous boundary conditions with respect to time. Thus one considers the problem
\begin{equation}\label{eq:firstCPLP}
\left\{
\begin{split}
(\partial_t - A_j)\psi_j(t_j) &=f_j(t_j),\qquad t_j\in\bigsqcup_{j\in \Ie \cup \Ee } I_j\\
 \underline{\psi}_- - \B \underline{\psi}_+&=\gu
\end{split}
\right.
\end{equation}
for given $\gu\in \Ke$ and $f\in L^2(\Ge,\au;\cX)$.
For $\B=0$ this corresponds to the usual initial condition $\psi(0)=\psi_{0}$.

The solution to this problem can be computed using the Green's function, where
-- as for ordinary differential equations with inhomogeneous boundary conditions -- the \textit{Lagrange-identity} \eqref{intbp} plays an important role. We start with a heuristic argument.
Integration by parts yields
\begin{align*}
\int_{\Ge} \left[(\partial_s r(t,s;A_\Ee)) \psi(s) + r(t,s;A_\Ee)\partial_s \psi(s)\right] ds 
= [r(t,\cdot;A_\Ee) \psi(\cdot)]_{\partial \Ge}
\end{align*} 
where 
\begin{align*}
[r(t,\cdot;A_\Ee) \psi(\cdot)]_{\partial \Ge}
= (r(t,a_j;A_\Ee) \psi(a_j) - r(t,0;A_\Ee) \psi(0))_{j\in \Ie}.
\end{align*}
Due to the properties of the Green's function
\begin{align*}
	\int_{\Ge} (\partial_s r(t,s;A_\Ee)) \psi(s) ds & = \int_{\Ge} r(t,s;A_\Ee)(-A_\Ee)\psi(s) ds \\ \int_{\Ge} (\partial_t r(t,s;A_\Ee)) \psi(s) ds &= \int_{\Ge} A_\Ee r(t,s;A_\Ee)\psi(s) ds.
\end{align*} 
Hence 
\begin{multline*}
\int_{\Ge} (\partial_s r(t,s;A_\Ee)) \psi(s) + r(t,s;A_\Ee)\partial_s \psi(s) ds 
= \int_{\Ge} A_\Ee^{-1}((\partial_t-A_\Ee) r(t,s;A_\Ee)) (-A_\Ee\psi(s)) \\+ r(t,s;A_\Ee)(\partial_s -A_\Ee) \psi(s) 
- A_\Ee^{-1}A_\Ee r(t,s;A_\Ee)) (-A_\Ee\psi(s))
- r(t,s;A_\Ee)A_\Ee \psi(s) 
ds \\ 
= -\psi(t) + \int_{\Ge}r(t,s;A_\Ee) f(s) ds,
\end{multline*} 
where one uses that
\begin{align*}
\int_{\Ge} (\partial_t -A_\Ee)r(t,s;A_\Ee) \psi(s) ds = \psi(t) \quad \hbox{and} \quad (\partial_s -A_\Ee) \psi(s) = f(s),
\end{align*} 
assuming that $\psi$ solves the Cauchy problem.
Hence
\begin{align*}
\psi(t)= \psi_{0}(t) +\int_{\Ge} r(t,s;A_\Ee) f(s) ds, \quad \psi_0(t)=-[r(t,\cdot;A_\Ee)\psi(\cdot)]_{\partial \Ge},
\end{align*} 
where $ \psi_0$ is given more explicitly by
\begin{align*}
 \psi_0(t)
% &= r(t,0;A_\Ee) \underline{\psi}_- - r(t,\au;A_\Ee) \underline{\psi}_+ \\
 = e^{A\tu} \underline{\psi}_- + e^{\tu A} (\mathds{1}-\B \, e^{\au A})^{-1}\B e^{\au A}\underline{\psi}_- - e^{\tu A} (\mathds{1}-\B \, e^{\au A})^{-1}\B \underline{\psi}_+.
\end{align*} 
For $\gu:= \underline{\psi}_- -\B\underline{\psi}_+$
one obtains $\psi_0(t)
= e^{A\tu} [\mathds{1} + (\mathds{1}-\B \, e^{\au A})^{-1}\B e^{\au A}]\gu = e^{\tu A}(\mathds{1}-\B \, e^{\au A})^{-1}\gu$.

\begin{theorem}\label{thm:main}
Let Assumption~\ref{ass:D1L2_2} be fulfilled and let $\B$ be compatible with $\Ke_A$. If $(\mathds{1}-\B \, e^{\au A})\vert_{\Ke_A}$ is boundedly invertible in $\Ke_A$,
then for any $\gu\in \Ke_A$ and $f\in L^2(\Ge,\au;\cX)$ there is a unique solution
		\begin{align*}
		\psi\in W^{1,2}(\Ge;\au;\cX) \cap L^2(\Ge,\au;D(A_\Ee))
		\end{align*}
		 to \eqref{eq:firstCPLP}.
The solution is given by
		\begin{equation}\label{eq:psivariat}
		\psi(t) = \psi_0(t)+\int_{\Ge} r(t,s;\B,A_\Ee)f(s) ds,\qquad t\in (\Ge,\au)
		\end{equation}
		where the kernel $r(\cdot,\cdot;\B,A)$ is given in~\eqref{eq:ker-r}, and 
		\begin{align*}
		\psi_0(t)
		:= e^{\tu A}(\mathds{1}-\B \, e^{\au A})^{-1}\gu,\qquad t\in (\Ge,\au).
		\end{align*} 
		In particular there exists a constant $C$ independent of $f$ and $\gu$ such that
		\begin{align*}
		\norm{D_t\psi}_{L^2(\Ge,\au;\cX)} + \norm{A_\Ee\psi}_{L^2(\Ge,\au;\cX)} \leq C (\norm{f}_{L^2(\Ge,\au;\cX)} + \norm{\gu}_{\Ke_A}).
		\end{align*}
\end{theorem}

\begin{proof}
Note that $\psi_0\in \in W^{1,2}(\Ge;\au;\cX) \cap L^2(\Ge,\au;D(A_\Ee))$ since $(\mathds{1}-\B \, e^{\au A})^{-1}\gu\in \Ke_A$ and all $-A_j$ are sectorial, cf.\  
 \cite[\S~3.4]{PruessSimonett} and in particular \cite[Prop. 3.4.2]{PruessSimonett} which can be adapted to finite intervals. Its traces satisfy
\begin{align*}
\underline{\psi_0}_-&= \gu + (\mathds{1}-\B \, e^{\au A})^{-1}\B e^{\au A}\gu,\\
\underline{\psi_0}_+&=
e^{\au A} \gu + e^{ \au A} (\mathds{1}-\B \, e^{\au A})^{-1}\B e^{\au A}\gu,
\end{align*}
 and hence $\underline{\psi_0}_- -\B\underline{\psi_0}_+=\gu$,
   and $(\psi_0)_j$ solves $(\partial_{t_j}-A_j)(\psi_0)_j=0$ on each edge $j\in \Ie$.

To prove uniqueness assume that there is another solution $\psi_0'$ to \eqref{eq:firstCPLP} in the solution space,
and consider the difference
$
\psi_{\delta}:=\psi_0'-\psi_0
$ which
 solves due to the linearity of the equation
\begin{equation*}
\left\{
\begin{split}
(\partial_t - A_j)(\psi_\delta)_j &=0,\qquad t_j\in\bigsqcup_{j\in \Ie \cup \Ee } I_j, \quad j\in \Ie \cup \Ee,\\
\underline{\psi_\delta}_- - \B \underline{\psi_\delta}_+&=0.
\end{split}
\right.
\end{equation*}
Because of the former equation, there exists $\gu'\in \Ke_A$ such that $\psi_{\delta}=e^{tA}\gu'$, while the latter implies
\begin{align*}
(\mathds{1}-\B e^{\au A})\gu'=0, \end{align*}
 and by the invertibility of $\mathds{1}-\B e^{\au A}$ it follows that $\psi_\delta\equiv 0$.
Hence the inhomogeneous boundary value problem is uniquely solvable with $\psi_0$ in the maximal $L^2$-regularity class, and the rest of the statement follows from Proposition~\ref{prop:resolvent}.
\end{proof}

\begin{remark}
	\begin{itemize}	
		\item[(a)] The solution formula given in Theorem~\ref{thm:main} is a generalization of the well-known variation of constants formula. Considering only one interval $[0,a]$ with boundary conditions $u(0)=0$, i.e. $\B=0$, we find that $\psi_0(t) = e^{tA} \underline{\psi}_{0}$ and $r(t,s;\B,A_\Ee)=r_0(t,s;A_\Ee)$.
  \item[(b)] Another classical case are periodic boundary conditions. For one interval $[0,a]$ with $u(0)=u(a)$, i.e., $\B=1$.
		\item[(c)] The solution to the time-graph Cauchy problem certainly satisfies a semigroup law on each edge. 
	\end{itemize}
\end{remark}

\subsection{Mapping properties}
Assume that $X_j=L^2(S_j,\mu_j)$ for some measure space $(S_j,\Sigma_j,\mu_j)$ are spaces of complex valued functions, and denote by $X_{j,\R}$ the cone of real valued functions. This induces spaces $\Ke_{\R}$ and $\Ke_{\R}^2$. 
\begin{proposition}\label{prop:mapp}
	Let the assumptions of Theorem~\ref{thm:main} be satisfied and let $X_j=L^2(S_j,\mu_j)$ be Hilbert spaces of complex valued functions.
	\begin{itemize}
		\item[(a)] If $\B$ and the operator families $(e^{t_jA_j})_{t_j\in I_j}$ leave $\Ke_{\R}$ invariant, 
		then the solution $\psi$ in Theorem~\ref{thm:main} is real for real data $\gu$ and $f$.
		\item[(b)] If in addition to $(a)$ the operators $\B$, $(\mathds{1}-\B^{\au A})^{-1}$ and the operator families $(e^{t_jA_j})_{t_j\in I_j}$ are positivity preserving, then the solution $\psi$ in Theorem~\ref{thm:main} is positive for all times $t\in (\Ge,\au)$ provided the data $\gu$ and $f$ are positive.
		\item[(c)] If the operator families $(e^{t_jA_j})_{t_j\in I_j}$ 
	 as well as $\B$ and $(\mathds{1}-\B \, e^{\au A})^{-1} $ are $L^{\infty}$-bounded, then the solution operator in Theorem~\ref{thm:main} is $L^{\infty}$-bounded. The solution operator defined by $\psi_0$ is $L^\infty$-contractive whenever so are the operator families $(e^{t_jA_j})_{t_j\in I_j}$ and $(\mathds{1}-\B \, e^{\au A})^{-1} $. If additionally $\B$ and $(\mathds{1}-\B \, e^{\au A})^{-1} $ are $L^1$-bounded, then the solution operator extrapolates to all $L^p$-spaces.
	\end{itemize}
\end{proposition}

\begin{proof}
We have shown in Proposition~\ref{prop:resolvent} that the Green's function is given by
$r_0(\cdot,\cdot;A)+r_1(\cdot,\cdot;\B,A)$.
It is apparent that the claimed properties for the solution to~\eqref{eq:firstCPLP} are proved as soon as corresponding properties hold for both $r_0(\cdot,\cdot;A),r_1(\cdot,\cdot;\B,A)$, and $\psi_0$ where the corresponding properties of $r_0$ are covered by the classical theory. Now, $r_1$ can be studied using its factorization into operators that also enjoy the corresponding properties.  For the mapping properties of $\psi_0$ analogous arguments apply.
\end{proof}

\subsection{Maximal $L^p$-regularity}\label{sec:maxlp}
For notational and mathematical simplicity, we have focused on the Hilbert space case and on maximal $L^2$-regularity.  In the case of evolution equations on $\mathbb R_+$, under the assumptions of Proposition~\ref{prop:mapp}.(c) the semigroup $e^{tA}$ extrapolates to a $C_0$-semigroup on all $L^p$-spaces, $p\in [1,\infty)$; this semigroup is additionally analytic  on $L^p$, $p\in [1,\infty)$, if $e^{tA}$ satisfies Gaussian estimates. By a celebrated result in~\cite{HiePru97} this implies in turn $L^p$-maximal regularity for $p\in (1,\infty)$, but our theory does not seem to allow us to discuss kernel estimates.
However, the solution formulae~\eqref{eq:ker-r} and~\eqref{eq:psivariat} suggest a straightforward generalization to the general case of maximal $L^p$-regularity in Banach spaces.

To this end, let $X_j$ be Banach spaces and $p\in (1,\infty)$. Consider the trace space
\begin{align*}
\Ke_{A,p}:= 
(X,D(A_\cV))_{1-1/p,p}= \bigoplus_{j\in \Ie} (X_j,D(A_j))_{1-1/p,p},
\end{align*} 
and collect the following assumptions.

\begin{assumption}\label{ass:Lpmaxreg}
Assume that $\Ee=\emptyset$, $X_j$ be Banach spaces of class UMD, and $p\in (1,\infty)$.
Suppose
that $-A_j$ are $\cR$-sectorial operators in $X_j$ of angle smaller than $\pi/2$, 
and that $\B\in \cL(\Ke)$. 
\end{assumption}

\begin{proposition}[Maximal $L^p$-regularity]\label{prop:Lpmaxreg} 
		Let the Assumption~\ref{ass:Lpmaxreg} be fulfilled, $\B$ be compatible with $\Ke_{A,p}$.  If $\mathds{1}-\B \, e^{\au A}$ is boundedly invertible in $\Ke_{A,p}$, 
		then for any $\gu\in \Ke_{A,p}$ and $f\in L^p(\Ge,\au;\cX)$ there is a unique solution
		\begin{align*}
		\psi\in W^{1,p}(\Ge,\au;\cX) \cap L^p(\Ge,\au;D(A)).
		\end{align*}
		 to \eqref{eq:firstCPLP}, where the solution is given by the same formulae as in Theorem~\ref{thm:main}
and
		\begin{align*}
		\norm{D_t\psi}_{L^p(\Ge,\au;\cX)} + \norm{A_\Ee\psi}_{L^p(\Ge,\au;\cX)} \leq C (\norm{f}_{L^p(\Ge,\au;\cX)} + \norm{\underline{d}}_{\Ke_{A,p}}).
		\end{align*}
\end{proposition}
\begin{proof}
The unperturbed part of the Green's function $r_0(\cdot,\cdot; A)$ defines a bounded operator 
\begin{align*}
L^p(\Ge,\au;\cX) \rightarrow D_t(\B)\cap L^p(\Ge,\au;D(A)), \quad \hbox{where }\B=0.
\end{align*}
It remains to verify that the correction terms have the same mapping properties. First,
\begin{align*}
	\psi_0(t)
		= e^{tA}(\mathds{1}-\B \, e^{\au A})^{-1}\gu,
\end{align*} 
and by assumption $(\mathds{1}-\B \, e^{\au A})^{-1}\gu\in \Ke_{A,p}$, and hence $\psi_0$ lies in the maximal regularity space, cf.\ \cite[Prop. 3.4.2]{PruessSimonett} which can be adapted to finite intervals.
Secondly,
	\begin{align*}
	 r_1(t,s;\B,A) = e^{\tu A} (\mathds{1}-\B \, e^{ A \au})^{-1}\B e^{(\au-\su)A}.
	\end{align*} 
Using that $\int_{\Ge}e^{(\au-\su)A} f(s) ds\in \Ke_{A,p}$ which follows from the classical variation of constants formula and maximal $L^p$-regularity of the initial value problem, it follows that $\int_{\Ge}r_1(t,s;\B,A) f(s)d$
 is in the maximal $L^p$-regularity space. 
 
 The proofs of Proposition~\ref{prop:resolvent} and Theorem~\ref{thm:main} now carry over to the present situation.
\end{proof}

\begin{remark}[Transference principle]
Transference principles which relate maximal $L^p$-regularity for the initial value problem to the maximal $L^p$-regularity problem with time-periodicity on
 the real line are well-established. 
Here, let $\Ee=\emptyset$ and $P(\B)$ for $\B=0$ have maximal $L^p$-regularity, then $P(\B)$ has maximal $L^p$-regularity for any $\B\in \cL(\Ke)$ satisfying the assumption of Proposition~\ref{prop:Lpmaxreg}. 
\end{remark}

\subsection{Regularity and other notions of solutions}\label{subsec:mild}
So far, we have focused on solvability in maximal regularity spaces, since these fit into a suitable functional analytic framework. 
Note that
the solution formula from Theorem~\ref{thm:main}  can be made sense of even under milder assumptions. Consider the case where $\Ee=\emptyset$, $X_j$ are Banach spaces, the operators $A_i$ generate $C_0$-semigroups in $X_i$, and
$\B\in \cL(\Ke)$. 
Classical results from semigroup theory carry over as long as sufficient compatibility of $\B$ is assumed. For instance, {smoother inhomogeneous transmission} data improve the regularity of solutions.

\subsubsection{Mild solutions}
Under the assumptions that
\begin{align*}
(\mathds{1}-\B \, e^{\au A})^{-1}\in \cL(\Ke), \quad \gu\in \Ke, \quad\hbox{and } f\in L^1(\Ge,\au\;\cX)
\end{align*}
the function defined by~\eqref{eq:psivariat} is a mild solution on each edge,
i.e., $\psi_j\in C(\overline{I_j};X_j)$ for each $j\in \Ie$, cf.
\cite[Prop.~1.3.4]{AreBatHie10} or \cite[Prop.~VI.7.4 and Prop.~VI.7.5]{EngNag00}), 
and the boundary conditions are attained in the larger space $\Ke$ while in Theorem~\ref{thm:main} they existed even in $\Ke_A$.

\subsubsection{Classical solutions} 
For $\B$ being compatible with $D(A_\cV)$, the conditions
\begin{align*}
(\mathds{1}-\B \, e^{\au A})^{-1}&\in \cL(D(A_\cV)), \quad \gu\in D(A_\cV), \quad\hbox{and }  \\
f=\{f_j\}_{j\in \Ie} \hbox{ with }  f_j&\in W^{1,1}(I_j;\cX_j) \hbox{ or } f_j\in C^{0}(\overline{I_j};\cX_j) \hbox{ for } j\in \Ie,   
\end{align*}
respectively, imply that the solution is classical on each edge, i.e., continuously differentiable with respect to time, cf.\ 
 \cite[Cor.~VI.7.6]{EngNag00} and \cite[Cor.~VI.7.8]{EngNag00}.
Of course there are many refinements of the classical semigroup theory which one can carry over to time-graphs by assuming sufficient compatibility between $\B$ and  the inverse of $\mathds{1}-\B \, e^{\au A}$.

\section{Iterative solvability}\label{sec:sym}\label{sec:iterati}

A time-graph Cauchy problem is iteratively solvable if it reduces to a finite sequence of initial value problems. 
This is made precise in the following definition.

\begin{definition}[Iterative solvability]
	Assume that $\Ee=\emptyset$ and $|\Ie|=n$, and that there exists an ordering of the edges $i_1, \ldots, i_n$ such that the solution to \eqref{eq:firstCP} satisfies
	\begin{align*}
	\partial_t \psi_n -A_n \psi_n = f_n, \quad \psi_n(0)-\B_{nn}\psi_n(a)= g_n
	\end{align*}
	and for any $1\leq j\leq n-1$ there exists some linear function $\varphi_{j+1}$ such that
	\begin{align*}
	\partial_t \psi_{j} -A_{j} \psi_{j} = f_{j}, \quad \psi_{j}(0)-\B_{jj}\psi_j(a_j)= \varphi_{j}(\psi_{j+1}, \ldots, \psi_n,g).
	\end{align*}
	Then we say that the \eqref{eq:firstCP} is \textit{iteratively solvable as a sequence of Cauchy problems on intervals}.
	If $\B_{jj}=0$ for all $j=1, \ldots,n$, then \eqref{eq:firstCP} is \textit{iteratively solvable as a sequence of initial value problems}.
\end{definition} 

Iterative solvability can be traced back to the block structure of $\B$. 
\begin{proposition}[Characterization if iterative solvability]\label{prop:iterativesol}
	Let $\Ee=\emptyset$ and $\B\in\cL(\Ke_A)$. 
\begin{enumerate}
\item[(a)] 
 The Cauchy problem \eqref{eq:firstCP} is iteratively solvable as a sequence of Cauchy problems on intervals if and only if, up to permutation of the edges, $\B$ is block tri-diagonal, i.e., there exists an ordering of the edges $i_1, \ldots, i_n$ such that $\B_{ij}=0$ for $j>i$. 
 \item[(b)] The Cauchy problem \eqref{eq:firstCP} is iteratively solvable as a sequence of initial value problems if and only if, up to permutation of the edges, $\B$ is block tri-diagonal with diagonal zero.
\end{enumerate}
\end{proposition}
\begin{proof}
To prove (a) we start assuming that $\B$ is block tri-diagonal.
Then 
		\begin{align*}
		\begin{bmatrix}
		\psi_1(0) \\ \vdots \\ \psi_n(0)
		\end{bmatrix}
		-
		\begin{bmatrix}
		\B_{11} & \hdots & B_{1n}\\ 0 & \ddots & \vdots \\ 0 & 0 & \B_{nn}
		\end{bmatrix}
		\begin{bmatrix}
		\psi_1(a_1) \\ \vdots \\ \psi_n(a_n)
		\end{bmatrix} =
		\begin{bmatrix}
		g_1 \\ \vdots \\ g_n
		\end{bmatrix}.
		\end{align*}
		Hence $\psi_n$ is the solution to the Cauchy problem on $i_n$, and $\psi_j$ for $j=1, \ldots, n-1$ solves the Cauchy problem on $i_j$ with 
		\begin{align*}
		\psi_j(0)-\B_{jj}\psi_j(a)=g_j-\sum_{i=j+1}^n \B_{ji}\psi_i(a_i)=:\varphi_j(\psi_{j+1}, \ldots, \psi_n,g).
		\end{align*}
		Conversely, if \eqref{eq:firstCP} is iteratively solvable as a sequence of Cauchy problems on intervals, then there exists an ordering of the edges such that $\B_{n,j}=0$ for $j\neq n$, and since $\varphi_j$ depends only on $\psi_{j+1}, \ldots, \psi_n$ and $g$ one concludes that $\B_{ij}=0$ for $j>i$.
		
For (b) notice that on each step on has an initial value problem if and only if $\B_{jj}=0$ for all $j\in \Ie$. 		 
\end{proof}

\begin{remark}\label{rem:iterative_solvability}
Invertibility of $\mathds{1}-\B \, e^{\au A}$ is automatically satisfied under the assumptions of Proposition~\ref{prop:iterativesol}.(b), since
\begin{align*}
\mathds{1}-\B \, e^{\au A} =	
\begin{bmatrix}
		\mathds{1} & -\B_{12}e^{a_2A_2} & \ldots & -\B_{1n}e^{a_nA_n} \\ 
		 0 & \ddots & \ddots  & \vdots \\ 
		 \vdots & \ddots & \mathds{1}  & -\B_{(n-1)n}e^{a_nA_n}\\
		 0 & \cdots & 0  & \mathds{1}
		\end{bmatrix}.
\end{align*}
This was a crucial assumption in Theorem~\ref{thm:main}: the structure of the time-graph already implies unique solvability. 
Under the weaker assumptions of Proposition~\ref{prop:iterativesol}.(a), instead, invertibility of all $\mathds{1}-\B_{ii}e^{a_iA_i}$ has to be imposed additionally. 
\end{remark}

An oriented graph $(\Ge,\au)$ \textit{contains a directed loop} if there exists a sequence of edges $i_{1}, \ldots, i_m$ such that 
\[
\partial_+i_{m} = \partial_-i_{1},\quad \partial_+i_{1} = \partial_-i_{2},
\quad \ldots,\quad \partial_+i_{m-1} = \partial_-i_{m}.
\]
(We stress that this usage of the notion of loop is slightly different than in the literature on metric graphs in that we do not require the intersection of the loop's closure and its complement's closure (in the time graph) to be a singleton.
We say that a loop is \textit{reflected by boundary conditions} if  $\B_{i_j i_{(j+1)}}\neq 0$ for each $i_j \in \{i_{1}, \ldots, i_m\}$, where on sets $i_{m+1}:=i_1$.

\begin{corollary}[Loops prevent iterative solvability]
 Let $(\Ge,\au)$ contain a directed loop which is reflected by the boundary conditions.
 \begin{itemize}
 	\item[(a)] Then \eqref{eq:firstCP} is not iteratively solvable as a sequence of initial value problems.
 	\item[(b)] If the loop contains more than one edge, then 
 	\eqref{eq:firstCP} is not iteratively solvable as a sequence of Cauchy problems on intervals.
 \end{itemize} 
\end{corollary}
\begin{proof}
	To prove (b): By assumption $m\geq 2$ and $B_{1,2}, \ldots, B_{(m-1),m}\neq 0$ and $\B_{m,1}\neq 0$. In particular for any permutation of the edges $\pi$ one has $B_{\pi(1)\pi(2)}\neq 0$ and $\B_{\pi(m)\pi(1)}\neq 0$, and therefore $\B$ cannot be block tri-diagonal and the claim follows from Proposition~\ref{prop:iterativesol} (a).
	
	To prove (a) use using Proposition~\ref{prop:iterativesol} (b) and note that if $m=1$, then $B_{11}\neq 0$, and if $m\geq 2$ this follows already from (b). 
\end{proof}

\begin{remark}[Graph symmetries and symmetries of solutions]
Periodic functions on $\R$ clearly induce functions on $\mathbb S^1$: are there further symmetries which can be encoded into a time-graph? Many graphs have a natural symmetry which corresponds to a group structure. Given a map  
\begin{align*}
T\colon (\Ge,\au) \rightarrow (\Ge,\au),
\end{align*}
this induces a map on function spaces
\begin{align*}
\hat{T}\colon	\{f\colon (\Ge,\au) \rightarrow X \} \rightarrow 	\{f\colon (\Ge,\au) \rightarrow X \}, \quad f\mapsto f\circ T, 
\end{align*}
see~\cite[\S~8.2]{Mug14}.
	Let $\Gamma$ be a group of mappings acting on $(\Ge,\au)$. Assume that
	\begin{equation}\label{eq:loopsym}
	\hat{G} D(D_t(\B))= D(D_t(\B)) \quad \hbox{and} \quad 	D_t(\B)\hat{ G} = \hat{G} D_t(\B)\qquad\hbox{	for all }G\in \Gamma.	
	\end{equation}
(For example, the shift on a loop satisfies~\eqref{eq:loopsym} with respect to the derivative operator with periodic boundary conditions.) It seems that there are very few graphs whose automorphism group is an infinite Lie group: in most cases, the automorphism group is finite. 
Having this, for any $G\in \Gamma$, $\hat{G}$ commutes with the solution operator given in Theorem~\ref{thm:main}. Thus the symmetry is reflected by the solution.
\end{remark}

\section{Examples and applications}\label{sec:examples}
The case of a loop with phase shift have been discussed already in the introduction as small modification of the classical periodic case. Also, the tadpole-like graph has been discussed there. We now discuss some other cases depicted in Figure~\ref{Fig2}.

\subsection{Splitting of systems}
Take the graph consisting of three internal edges as in Figure~\ref{Fig2}.(b), and consider for sectorial spatial operators $-A_i$ in Hilbert spaces $X_i$, $i=1,2,3$, the problem $\partial_t\psi_i -A_i\psi_i =f_i$, $ i=1,2,3$,  with homogeneous boundary conditions
\begin{align*}
 \psi_1(0)=g_1, \quad \psi_2(0) = \B_{21} \psi_1(a_1), \quad \psi_3(0) = \B_{31} \psi_1(a_1).
\end{align*}
This corresponds to~\eqref{eq:firstCPLP} for
\begin{align*}
\B= \begin{bmatrix}
0 & 0 & 0 \\
\B_{21} & 0 & 0 \\
\B_{31} & 0 & 0
\end{bmatrix} \quad \hbox{and}\quad
g = \begin{bmatrix}
g_1 \\ 0 \\ 0 
\end{bmatrix}.
\end{align*}
If $\B$ is compatible with the trace space $\Ke_A$, one observes that 
\begin{align*}
\mathds{1}-\B e^{\au \, A}=
\begin{bmatrix}
\mathds{1} & 0 & 0\\
-\B_{21} e^{a_1 A_1} & \mathds{1} & 0\\
-\B_{31} e^{a_1 A_1} & 0 & \mathds{1}\\
\end{bmatrix}
\end{align*}
is  invertible for any $\B_{21}, \B_{32}$, cf.\ Remark~\ref{rem:iterative_solvability}. Therefore by Theorem~\ref{thm:main} a unique solution to this problem exists for all $\gu\in \Ke_A$; in particular $\gu=(g_1,0,0)^T \in \Ke_A$ means that $g_1\in [X_1,D(A_1)]_{1/2}$. 

The tree graph given in Figure~\ref{Fig2}.(f) results from an iteration of such as splitting procedure, where as above any splitting condition as long as it is compatible with the trace space is admissible. That such splitting problems can be solved iteratively as sequence of initial value problems is straight forward, it can also be seen more formally by applying
Proposition~\ref{prop:iterativesol}.

\subsection{Superposition of systems}
Analogously, take the graph consisting of three internal edges as in Figure~\ref{Fig2}.(c), and consider for sectorial spatial operators $-A_i$ in Hilbert spaces $X_i$, $i=1,2,3$, the problem $\partial_t\psi_i -A_i\psi_i =f_i$, $ i=1,2,3$,  with homogeneous boundary conditions
\begin{align*}
\psi_1(0)=g_1, \quad \psi_2(0) = g_2, \quad \psi_3(0)=\B_{31}\psi_1(a_1) + \B_{32}\psi_2(a_2),
\end{align*}
for $i=1,2,3$.
This corresponds to~\eqref{eq:firstCPLP} for
\begin{align*}
\B= \begin{bmatrix}
0 & 0 & 0 \\
0 & 0 & 0 \\
\B_{31} & \B_{32} & 0
\end{bmatrix}, \quad \gu = \begin{bmatrix}
g_1 \\ g_2 \\ 0 
\end{bmatrix},\quad \hbox{where }
\mathds{1}-\B e^{\au \, A}=\begin{bmatrix}
\mathds{1} & 0 & 0\\
0 & \mathds{1} & 0\\
-\B_{31} e^{a_1 A_1} & -\B_{32} e^{a_2 A_2} & \mathds{1}\\
\end{bmatrix}.
\end{align*}
Assuming that $\B$ is compatible with $\Ke_A$ one observes that $\mathds{1}-\B e^{\au \, A}$
is invertible in $\Ke_A$, cf.\ Remark~\ref{rem:iterative_solvability}.  So,  Theorem~\ref{thm:main} is applicable again, and there is a unique solution to this problem if $g_1\in [X_1,D(A_1)]_{1/2}$ and $g_2\in [X_2,D(A_2)]_{1/2}$. One observes that this is iteratively solvable as a sequence of initial value problems, too.

\subsection{Tadpole graph}
Consider two edges with
\begin{align*}
\psi_1(0)=\psi_1(a_1)\quad \hbox{and}\quad \psi_2(0)=\B_{21}\psi_1(a_1)\quad \hbox{corresponding to } \B = \begin{bmatrix}
\mathds{1} & 0 \\
\B_{21}& 0
\end{bmatrix}.
\end{align*}
Note that $\mathds{1}-\B e^{\au A}$ is invertible if and only if $\mathds{1}-e^{a_1 A_1}$ is invertible. So, solvability is assured if for instance $\norm{e^{a_1 A_1}}<1$, i.e., the semigroup generated by $A_1$ is contractive and exponentially decaying.
This tadpole graph system can be interpreted as a time-periodic system where the output is used as initial data for a new system.
In the notion introduced in Section~\ref{sec:iterati}, this means the problem can be solved iteratively, first solving a time-periodic problem and then an initial value problem, the data of which depend on the solution obtained in the first step.
 
\subsection{Frequency dependent couplings}\label{subsec:Frequency dependent couplings}
Frequency-dependent transition conditions between time branches may also be considered. Assume for simplicity that $A_1=A_2$ are positive self-adjoint operators with discrete spectrum $k_1\leq k_2\leq \ldots$ (counted with multiplicities): any element in $\Ke$ can thus be expanded in terms of eigenfunctions.
For two edges one can e.g.\ consider the left shift operator $S_-$ defined by
\begin{align*}
\psi = \sum_{n\in \N} a_n \psi_n \mapsto S_-\psi:=\sum_{n\in \N} a_{n+1} \psi_n,
\end{align*}
where $(\psi_n)$ is an orthonormal basis of eigenfunctions: this induces a map also in $D(A^{1/2})$. Then
\begin{align*}
\B = \begin{bmatrix}
0 & 0 \\ 
0 & S_-
\end{bmatrix}
\end{align*}
is an admissible transmission condition where the first row induces an initial condition on $\psi_1(0)$ and the second is the frequency shift.
One may also consider a projection $P_I$ onto certain frequency ranges $I\subset \N$,
\begin{align*}
P_{I}\psi := \sum_{n\in I} a_n \psi_n
\end{align*}
and one could have a splitting of the system 
\begin{align*}
\psi_1(0)=0, \quad \psi_2(0)=P_{I_2}\psi_1(a_1), \quad \psi_3(0)=P_{I_3}\psi_1(a_1), \quad \psi_4(0)=P_{I_4'}\psi_2(0) + P_{I_4''}\psi_3(0)
\end{align*}
corresponding to
\begin{align*}
\B = \begin{bmatrix}
0 & 0 & 0 & 0 \\
P_{I_2} & 0 & 0 & 0 \\
P_{I_3} & 0 & 0 & 0 \\
0 & P_{I_4'} & P_{I_4''} & 0
\end{bmatrix}.
\end{align*}
This is iteratively solvable as sequence of initial value problems, and hence it is well-posed.

\subsection{Lions maximal $L^2$-regularity problem for non-autonomous Cauchy problems}
Lions' maximal regularity problem for non-autonomous Cauchy problems
considers
\begin{align*}
(\partial_t - A(t))\psi(t)=f(t), \quad \psi(0)=\psi_0
\end{align*}
for  a Hilbert space $X$ and $f\in L^2(0,T;X)$ and asks if the solution satisfies $\psi\in W^{1,2}(0,T;X)$; this would in turn imply that also $A(\cdot)\psi\in L^2(0,T;X)$. This problem has a long history and much remarkable work has been devoted to it. More precisely, depending on $A(\cdot)$ there are counterexamples as well as criteria which assure an affirmative answer: we refer the interested reader to~\cite{HieMon00} for an early study of maximal regularity for non-autonomous problems, and to
\cite{AreDieFac17} for further information and updated references.

The particular case of $A(\cdot)$ being a (matrix-valued) step function with matching trace spaces 
\begin{align*}
A(t)=\begin{cases}
A_1, & t \in (0,t_1), \\
\ldots & \\
A_n, & t \in (t_{n-1},t_n), \\
\end{cases} \quad \hbox{with} \quad D(A_j^{1/2})=D(A_i^{1/2})\hbox{ for all } 1\leq i,j\leq n.
\end{align*} has already been used by \cite{Laasri2016} as a first step to consider $A(\cdot)$ which are of bounded variation. The time-graph approach does not give any additional information at this point but it underlines the role of the compatibility assumption for the trace spaces.
Using our approach we directly see that the corresponding abstract time-graph Cauchy problem can be studied by means of 
\begin{align*}
\B = \begin{bmatrix}
0 & \cdots & \cdots & \cdots& 0 \\
1& \ddots& \ddots & \ddots & \vdots \\
0 & \ddots& \ddots & \ddots & \vdots \\
\vdots& \ddots & \ddots & \ddots & \vdots \\
0& \cdots & 0 & 1 & 0 
\end{bmatrix},
\end{align*}
hence it is iteratively solvable by Proposition~\ref{prop:iterativesol}.

\subsection{Outline on non-parabolic Cauchy problems}\label{sec:otherCPs}
So far, the focus has been on parabolic Cauchy problems, but the Green's function \textit{Ansatz} makes sense also in some non-parabolic situations. 

\subsubsection{Schr\"odinger equation}
Let us study the Schrödinger-type problem
\begin{equation}\label{eq:firstCP_Schroedinger}
\left\{
\begin{split}
(\partial_t - iA_j)\psi_j(t_j) &=f_j(t_j),\qquad t_j\in\bigsqcup_{j\in \Ie \cup \Ee } I_j,\\
 \underline{\psi}_- -\B \underline{\psi}_+ &=\gu.
\end{split}
\right.
\end{equation}
Provided that $\mathds{1}-\B \, e^{i\au A}$ is  invertible  in $\Ke$, the solution map 
\begin{align*}
S_{\B}(t) 
		= e^{itA}(\mathds{1}-\B \, e^{i\au A})^{-1}
\end{align*}
given by $\psi_0$ in Theorem~\ref{thm:main} is well-defined for all $\gu\in \Ke$ and defines a mild solution. (Notice that  invertiblility in $\Ke$ is sufficient since here we merely aim at mild solutions.) 
\begin{remark}
While the time-graph $\Ge$ need not display a group structure and, therefore, the issue of time-reversal need generally not be well-defined, we may still wonder whether the family of solution operators that govern ~\eqref{eq:firstCP_Schroedinger} consists of unitary operators.
%\begin{proposition}\label{prop:unitary_Schroedinger}
	Assume that $A$ and $\B\in \cL(\Ke)$ be self-adjoint such that $e^{i\au A}$ and $\B$ commute, and $\mathds{1}-\B \, e^{i\au A}$ is invertible. Then 
	the solution operators to the time-graph Schr\"odinger equation	\eqref{eq:firstCP_Schroedinger} are unitary, 	 i.e., 
\begin{align*}
S_{\B}(t) S_{\B}(t)^* = S_{\B}(t)^* S_{\B}(t) = \mathds{1},
\end{align*}
if and only if $\B^2 = 2\B\cos(\underline{a}A)$.
%\end{proposition}
%hence one needs that
%in order to have a unitary solution map.}
This follows, using that all operators commute, from
\begin{align*}
e^{itA}(\mathds{1}-\B \, e^{i\au A})^{-1} (\mathds{1}- e^{-i\au A}\B)^{-1} e^{-itA}= 
(\mathds{1}-\B \, e^{i\au A})^{-1} (\mathds{1}- \B e^{-i\au A})^{-1}=
\mathds{1}
\end{align*}
which is in turn equivalent to
$\B^2 = 2\B\cos(\underline{a}A)$.
If, in particular, $\B$ is invertible, then the above condition amounts to $\B=2 \cos( \au A)$. Accordingly, 
there is a unitary  solution operator for fixed jump condition, i.e., $\B=\mathds{1}$, if and only if $\sigma(\au A) \subset \pi/3+2\pi\Z$. So, the classical case $\B=0$ is not the only case of a unitary  solution operator. Note that time inversion is still possible even for non-unitary solution operator, but the time-reversed dynamics differs from the time-forward dynamics and going forth and back is not necessarily equal to stay at one time. 
\end{remark}

\subsubsection{Second order Cauchy problem}
The above setting can be generalized to different kind of evolution equations.
Let $\B_1,\B_2$ be bounded linear operators on $\mathcal K^2$ and consider the second order Cauchy problem
\begin{equation}\label{eq:secCP}
\left\{
\begin{split}
(\partial^2_t - A_j)\psi_j(t_j) &=f_j(t_j),\qquad t_j\in\bigsqcup_{j\in \Ie \cup \Ee } I_j,\\
 \B_1 \underline{\psi}_+ -\underline{\psi}_- &=\gu_1\\
 \B_2 \underline{\psi_t}_+ -\underline{\psi_t}_- &=\gu_{2}.
\end{split}
\right.
\end{equation}
The idea is to decompose this into
\begin{align*}
(\partial_{t}^2 -A)\psi = (D_t(\B_1) - i|A|^{1/2})(D_t(\B_2) + i|A|^{1/2})\psi,
\end{align*}
where $A$ is skew-symmetric. Hence one has to solve two first order problems iteratively, where
%\begin{proposition}\label{thm:main2}
assuming in addition to the assumptions of Theorem~\ref{thm:main} that $A_j$ for $j\in \Ie$ are self-adjoint and boundedly invertible. Then for any $\gu_1, \gu_2\in \Ke$ there is a unique solution to \eqref{eq:secCP}.
%\end{proposition}	

\subsection{Outline on mixed order systems} It is also possible to discuss evolution equations whose nature is different on each time branch; in particular, it is possible to define an operator on $\T$ which agrees with a first derivatives on a subset of $\T$ and with a second derivative on the remaining time branches. Defining appropriate transition conditions is, however, less obvious: a thorough discussion of ``well-behaved''
		transition conditions can be found in~\cite{HusMug13}.
Following these lines one can use the Kalton--Weis approach to solve a Cauchy problem of the type		
		 \begin{align*}
	(-\partial_t -A_1)\psi_1&=0, \quad \hbox{on} \quad [0,a_1], \\
	(\partial_{tt} - A_2)\psi_2&=0, \quad \hbox{on} \quad [0,a_2].
\end{align*}
Focusing on this example we consider operators $A_1,A_2$ in Hilbert spaces $X_1,X_2$
and couplings defined by 
\begin{align*}
(P+L)\underline{\psi} + P^{\perp} \underline{\psi}'=0, \quad
 \underline{\psi}=\begin{bmatrix}
2^{-1/2}(\psi_1(a_1) + \psi_1(0)) \\
\psi_2(a_2) \\
\psi_2(0) \\
\end{bmatrix}, 
\underline{\psi}'=\begin{bmatrix}
2^{-1/2}(-\psi_1(a_1) + \psi_1(0)) \\
-(\psi_2)_t(a_2) \\
(\psi_2)_t(0) \\
\end{bmatrix}
\end{align*}
for an orthogonal projection $P$ in $\Ke:= X_1\oplus X_2$, $P^{\perp}=\mathds{1}-P$ and $L\in \cL(\Ke)$ with $LP^{\perp}=P^{\perp}L$. With these couplings the operator $\T_{P,L}$ defined on $L^2(0,a_1;X_1)\oplus L^2(0,a_2;X_2)$ by
\begin{align*}
\begin{split}
&\T_{P,L} \psi_1 = -\psi_1', \quad \T_{P,L} \psi_2 = \psi_2'', \\ &D(\T_{P,L}):= \{\psi_1\in W^{1,2}(0,a_1;X_1), \psi_2\in W^{2,2}(0,a_1;X_1)\colon (P+L)\underline{\psi} + P^{\perp} \underline{\psi}'=0 \}
\end{split}
\end{align*} 
is m-dissipative if $-L$ is dissipative, see \cite[Theorem 4.1]{HusMug13}, and similarly to Corollary~\ref{cor:Hinfty} one concludes that it has a bounded $H^\infty$-calculus of angle $\pi/2$. If the spatial operator $-A$ is sectorial and commutes with the boundary conditions, the one can apply the Kalton--Weis theorem to obtain well-posedness in a maximal regularity space. A Green's function approach could be pursued as well on the lines of the Green's function from \cite[Proposition 6.6]{HusMug13}.

\section{A tentative interpretation of time-graphs}\label{sec:timetravel}

%In the case of linear time, not only the group property of $\R$ plays a role. 
Several convenient properties of semigroups depend decisively on the order structure of the underlying set, so it is conceivable to relax the standard approach to time evolution and study semigroup-like operator families that are indexed on posets different from $\R_+$. 

One particularly simple case is that of a tree-like time structure. More precisely, we 
%want to 
allow for a time that looks like a rooted tree, see Figure~\ref{Fig2}.(f). This seems to be conceptually very close to H.\ Everett's many worlds interpretation of quantum mechanics~\cite{DewGra15}, but our mathematical theory is not restricted to Schrödinger-type equations, see Section~\ref{sec:Green}, and a precise analysis of similarities and differences with Everett's interpretation goes far beyond the scope of this note.
%Tree graphs come to one's mind when 
%thinking of the multiverse interpretation of quantum mechanics.
In a very simplified synopsis 
the many worlds interpretation claims in order to conciliate probabilistic and deterministic interpretations of quantum mechanics that time splits at each point in time and each possible state is actually attained in one of the parallel universes.

Imagining parallel-universes, one would have no evidence of these in the case of a tree graph. Only if there is some interaction between different `time-branches', then one can know of the other now non-parallel but interacting universes. 
This leads to an interpretation of time-travel in terms of time-graphs, and it seems that time-graphs are a convenient way of picturing to oneself time-travel independently of its actual physical meaning.

Note that in the case of tree graphs, an evolutionary system can be solved iteratively starting with the first edge where some initial condition is imposed, one determines the state at the end of the edge which is the used as new initial conditions for the next level of the tree etc.   Interesting problems arise when oriented loops are allowed, as outlined in Section~\ref{sec:iterati}. Then the time-evolution can no longer be described iteratively and there is no clear direction distinguishing `before' and `after'. This is one of the problems occurring in the scientific interpretation of closed time-like curves in general relativity, cf.\ \cite[Chapter 5]{Hawking1973} for a discussion of closed time-like curves occurring in exact solutions to the Einstein equations.
%\footnote{\DM{Der/die Referee fragt: ``What exactly in \cite{Hawking1973} are you referring to?''. Vermutlich solltest Du lieber antworten.}}
 and much earlier this has spurred the imagination of science fiction authors.
  Just to mention a few, there are H.G.\ Wells 'novel \textit{The time machine} (1895) as well as 
  \textit{The Man in the High Castle} (1962) and further works written by Philip K.\ Dick between the 1950s and the 1970s, whose main theme are interacting parallel universes. 
  Variations of these themes and particular the so-called \textit{grandfather paradox} -- preventing one's one birth during a journey back in time or violating causality in some other way -- are at the origin of several mainstream movies, like the classic \textit{Back to the Future} (1985) or the more recent \textit{Looper} (2012), in whose plot a person is sent back from 2074 to 2044 so that he can meet and be killed by %Joseph Gordon-Levitt, 
   his younger self\footnote{At one point, he urges his younger self to refrain from theoretical considerations stating ``I don't want to talk about
   time travel [...]. If we start talking about it,
   we're gonna be here all day.
   Making diagrams with straws'': a witty allusion to time as a graph.}.
   A further, different narrative trick is that of time-loops, illustrated for instance in 
   the movies \textit{Groundhog Day} (1993)
   or \textit{Miss Peregrine's Home for Peculiar Children} (2016), which tell the stories of  Phil Connors -- respectively, of a group of children and their guardian -- who get trapped in a time-loop around February 2, 1992 -- respectively, September 3, 1943 --, before eventually managing to escape: in mathematical terms, this does not mean that a certain function is periodic (the days spent by the main characters in either loop are \textit{not} identical),  but rather that it \textit{only} satisfies a certain identity condition at different instants (Phil Connors' environment is reset every day at 6:00 am, and so is the environment of the peculiar children and their guardian, every day at 9:07 pm): this time development can be captured by Figure~\ref{Fig2}.(e) or Figure~\ref{Fig6}.(d). %Figure~\ref{Fig2}.(e)}. 
   	This suggests a much more down-to-earth interpretation of evolution on branching time structures: namely, it conveniently allows us to formalize the requirement that solutions at different time instants respect certain algebraic relations. 

  The question of whether such fictional situations can be reconciled with our deeply rooted perception of time as linear seems for us related to the problem of representing a time-graph Cauchy problem as a sequence of initial value problems which can be solved one after another. As we have pointed out in Section~\ref{sec:iterati} this is closely related to the absence of loops inside the time-graph, and in this case the solution operator acts in a truly global (in time) fashion, as the solution at some point depends on all other times including future-like times.

\subsection{Time travel, multiverses and the grandfather paradox}
A time-travel scenario can be depicted as in Figure~\ref{Fig2}.(g) with a link between some point in the future and a point which might be in the past.
Considering such a graph one can e.g. impose the transmission conditions
\begin{align*}
\psi_1(0)=\gu_1, \quad \psi_{2}(0) = \psi_{1}(a_1)+ \psi_{4}(a_4), \quad \psi_{3}(0) = \psi_{2}(a_2), \quad \psi_{4}(0) = \psi_{2}(a_2) 
\end{align*}
that correspond to
\begin{align*}
\B = \begin{bmatrix}
0 & 0 & 0 & 0 \\
1 & 0 & 0 & 1 \\
0 & 1 & 0 & 0 \\
0 & 1 & 0 & 0 \\
\end{bmatrix}, \quad g=\begin{bmatrix}
\gu_1 \\ 0 \\ 0 \\ 0 
\end{bmatrix}.
\end{align*}
Therefore, this is not solvable as sequence of Cauchy problems on intervals, but 
well-posed in the sense of Theorem~\ref{thm:main}
 under suitable assumptions on the spatial operators $A_i$. 
In a sense, time-travel occurs in this deterministic world, but there is no free will to cause a grandfather paradox: the system is forced to be contradiction-free. This  resembles the case of time-periodic solutions. Given a solution $\psi$ to 
\begin{align*}
\partial_t \psi - A\psi = f, \quad \psi(0)=\psi(1),
\end{align*}
one can compare this to the solution to
\begin{align*}
\partial_t \varphi - A\varphi = f, \quad \varphi(0)=\psi(0).
\end{align*} 
Due to the uniqueness of solutions the system comes back to its original state, i.e., $\psi=\varphi$ and $\varphi(1)=\psi(1)=\psi(0)$. Living in a time-periodic world, would thus be locally like living in a time-interval world with initial conditions, but nevertheless  globally it is time-periodic. Similarly, in the scenarios of Figure~\ref{Fig2}.(g) or Figure~\ref{Fig6}.(a), solutions have in time non-local constraints which are seen only on a global level.

Considering the graph from Figure~\ref{Fig2}.(h) one can e.g.\ impose the transmission conditions
\begin{align*}
\psi_1(0)=\gu_1, \quad \psi_{2}(0) = \psi_{1}(a_1), \quad \psi_{3}(0) = \psi_{2}(a_2), \quad \psi_{4}(0) = \psi_{2}(a_2), \quad \psi_{5}(0)=
\psi_{4}(a_4)+\psi_{1}(a_1)
\end{align*}
which leads to
\begin{align*}
\B = \begin{bmatrix}
0 & 0 & 0 & 0 & 0 \\
1 & 0 & 0 & 0 & 0 \\
0 & 1 & 0 & 0 & 0\\
0 & 1 & 0 & 0 & 0\\
1 & 0 & 0 & 1 & 0\\
\end{bmatrix}
, \quad g=\begin{bmatrix}
\gu_1 \\ 0 \\ 0 \\ 0 \\0 
\end{bmatrix}.
\end{align*}
This is solvable as sequence of initial value problems, because $B_{14}=0$ the loop in the graph is not reflected by boundary conditions. A grandfather paradox does not occur because we actually have a sequence of initial value problems, and this seems the way we represent time-travel in our thoughts when watching a science fiction movie: A time traveler reaches a past where the initial conditions are the actual state of the past \textit{plus} the time traveler. This mixer gives new initial conditions which lead to new events which actually do not affect the time from which the time traveler comes from. Traveling back to the present does not lead to any contradictions since one has a simple superposition of two time branches so that changes due to the altered time branch can be incorporated. This becomes more transparent when as in  Figure~\ref{Fig6}.(c) -- compared to Figure~\ref{Fig6}.(b) --  an auxiliary edge is inserted. 

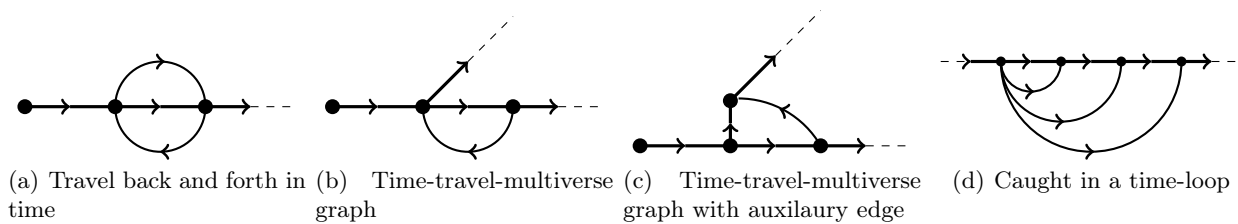
\begin{figure}[h]\label{fig:various_graphs}
	\begin{center}	
\subfigure[Travel back and forth in time]{	
	\begin{tikzpicture}[scale=0.6]
			\fill[black] (0,0) circle (1ex);
			\draw[->, black, very thick] (0,0) -- (1,0);
			\draw[black, very thick] (1,0) -- (2,0);
			\fill[black] (2,0) circle (1ex);
			\draw[->, black, very thick] (2,0) -- (3,0);
			\draw[black, very thick] (3,0) -- (4,0);
			\fill[black] (4,0) circle (1ex);
			\draw[->, black, very thick] (4,0) -- (5,0);
			\draw[black, dashed] (5,0) -- (6,0);
			\draw[->, black, very thick] (3.1,-1) -- (3,-1);
			\draw[thick] (4,0) arc (0:-180:1);
			%\draw[->, black, very thick] (2,0) -- (3,1);
			%\draw[black, dashed] (3,1) -- (4,2);
			\draw[thick] (4,0) arc (0:180:1);	
				\draw[->, black, very thick] (3,1) -- (3.1,1);
			\end{tikzpicture}
			}
\subfigure[Time-travel-multiverse graph]{	\begin{tikzpicture}[scale=0.6]
			\fill[black] (0,0) circle (1ex);
			\draw[->, black, very thick] (0,0) -- (1,0);
			\draw[black, very thick] (1,0) -- (2,0);
			\fill[black] (2,0) circle (1ex);
			\draw[->, black, very thick] (2,0) -- (3,0);
			\draw[black, very thick] (3,0) -- (4,0);
			\fill[black] (4,0) circle (1ex);
			\draw[->, black, very thick] (4,0) -- (5,0);
			\draw[black, dashed] (5,0) -- (6,0);
			\draw[->, black, very thick] (3.1,-1) -- (3,-1);
			\draw[thick] (4,0) arc (0:-180:1);
			\draw[->, black, very thick] (2,0) -- (3,1);
			\draw[black, dashed] (3,1) -- (4,2);	
			\end{tikzpicture}
		}
\subfigure[Time-travel-multiverse graph with auxilaury edge]{	\begin{tikzpicture}[scale=0.6]
			\fill[black] (0,0) circle (1ex);
			\draw[->, black, very thick] (0,0) -- (1,0);
			\draw[black, very thick] (1,0) -- (2,0);
			\fill[black] (2,0) circle (1ex);
			\draw[->, black, very thick] (2,0) -- (3,0);
			\draw[black, very thick] (3,0) -- (4,0);
			\fill[black] (4,0) circle (1ex);
			\draw[->, black, very thick] (4,0) -- (5,0);
			\draw[black, dashed] (5,0) -- (6,0);
		
			\draw[thick] (4,0) arc (30:90:2.1);
				\draw[<-, black, very thick] (3.09,0.82) -- (3.19,0.785);
\fill[black] (2,1) circle (1ex);		
		
			\draw[->, black, very thick] (2,0) -- (2,0.5);
			\draw[black, very thick] (2,0.5) -- (2,1);
		\draw[->, black, very thick] (2,1) -- (3,2);	
			\draw[black, dashed] (3,2) -- (4,3);	
			\end{tikzpicture}
		}	
	\subfigure[Caught in a time-loop]{	
	\begin{tikzpicture}[scale=0.4]
		\draw[black, very thick] (-1,0) -- (0,0);
		\draw[->, black, very thick] (-1.1,0) -- (-1,0);	
		\draw[black, dashed] (-2,0) -- (-1,0);
		\fill[black] (0,0) circle (1ex);
		\draw[->, black, very thick] (0,0) -- (1,0);
		\draw[black, very thick] (1,0) -- (2,0);
		\fill[black] (2,0) circle (1ex);
		\draw[->, black, very thick] (2,0) -- (3,0);
		\draw[black, very thick] (3,0) -- (4,0);
		\fill[black] (4,0) circle (1ex);
		\draw[->, black, very thick] (4,0) -- (5,0);
		\draw[black, very thick] (5,0) -- (6,0);
		\fill[black] (6,0) circle (1ex);
		\draw[<-, black, very thick] (1.1,-1) -- (1,-1);
		\draw[thick] (2,0) arc (0:-180:1);
		\draw[<-, black, very thick] (2.1,-2) -- (2,-2);
		\draw[thick] (4,0) arc (0:-180:2);
		\draw[<-, black, very thick] (3.1,-3) -- (3,-3);
		\draw[thick] (6,0) arc (0:-180:3);			\draw[->, black, very thick] (6,0) -- (7,0);
		\draw[black, dashed] (7,0) -- (8,0);
		\end{tikzpicture}
	}
	%	\caption{}\label{Fig7}
	\end{center}
		\caption{Time travel with and without parallel universes, and time-loop}\label{Fig6}
\end{figure}

\subsection{Caught in a time-loop}
The scenario that describes being caught in a time-loop just as in  \textit{Groundhog Day} can also be represented by time-graphs. Having a time-graph as in  Figure~\ref{Fig6}~(d) one can impose at the first vertex 
a splitting into several copies of the world. At each subsequent vertex there is a superposition of this original state the time evolution of the incoming edge, where the superposition is such that only the main character is replaced while all the rest goes back to the original state. Such conditions would allow for iterative solvability as a sequence of initial value problems.

Representing such a plot properly would, in addition, need a dynamical graph the development of which depends on the solution; one would also need to incorporate an end-condition stating that if the solution reaches a certain state, then the time evolution proceeds as a usual time axis. This would be a nonlinear feature. Typically, such an end-condition consists in the main character's reaching a certain goal or a key insight into the meaning of life.

Summarizing, it seems that our thinking is preassigned  to represent time as linear with well-specified past and future, and even when imagining  science fictional scenarios of time travel, time-loops and parallel universes we search for an ordering asking 'what happened first?', 'what happened then?', '$\ldots$ and then?' etc. So, in our language every science-fictional scenario needs a representation as a sequence of iteratively solvable sequence of initial value problems. The other way round, thinking of a properly time-periodic movie would be quite repetitive.

\bibliographystyle{alpha}
\bibliography{literatur}

\begin{thebibliography}{HKKT20}

\bibitem[AB02]{AreBu02}
W.~Arendt and S.~Bu.
\newblock The operator-valued {M}arcinkiewicz multiplier theorem and maximal
  regularity.
\newblock {\em Math.\ Z.}, 240:311--343, 2002.

\bibitem[AB04a]{AreBu04}
W.~Arendt and S.~Bu.
\newblock {Operator-valued Fourier multipliers on periodic Besov spaces and
  applications}.
\newblock {\em Proc.\ Edinburgh Math.\ Soc.}, 47:15--33, 2004.

\bibitem[AB04b]{AreBu04b}
W.~Arendt and S.~Bu.
\newblock {Operator-valued multiplier theorems characterizing Hilbert spaces}.
\newblock {\em J.\ Aust.\ Math.\ Soc.}, 77:175--184, 2004.

\bibitem[Abe97]{Abe97}
M.~Abe.
\newblock Time in {B}uddhism.
\newblock In S.~Heine, editor, {\em {Zen and Comparative Studies}}, Library of
  Philosophy and Religion, pages 163--169. MacMillan, London, 1997.

\bibitem[ABHN10]{AreBatHie10}
W.\ Arendt, C.J.K.\ Batty, M.\ Hieber, and F.\ Neubrander.
\newblock {\em Vector-{V}alued {L}aplace {T}ransforms and {C}auchy {P}roblems
  -- Second Edition}, volume~96 of {\em Monographs in Mathematics}.
\newblock Birkh{\"a}user, Basel, 2010.

\bibitem[ADF17]{AreDieFac17}
W.~Arendt, D.~Dier, and S.~Fackler.
\newblock {J. L. Lions' problem on maximal regularity}.
\newblock {\em Arch. Math.}, 109:59--72, 2017.

\bibitem[Ama95]{Ama95}
H.\ Amann.
\newblock {\em {Linear and quasilinear parabolic problems.\ Vol.\ 1: Abstract
  linear theory}}.
\newblock Birkh{\"a}user, Basel, 1995.

\bibitem[Are04]{Are04}
W.\ Arendt.
\newblock Semigroups and evolution equations: Functional calculus, regularity
  and kernel estimates.
\newblock In C.M.\ Dafermos and E.\ Feireisl, editors, {\em Handbook of
  Differential Equations: Evolutionary Equations -- Vol.\ 1}. North Holland,
  Amsterdam, 2004.

\bibitem[CK18]{CelKye18}
A.~Celik and M.~Kyed.
\newblock Nonlinear wave equation with damping: periodic forcing and
  non-resonant solutions to the {K}uznetsov equation.
\newblock {\em ZAMM Z. Angew. Math. Mech.}, 98(3):412--430, 2018.

\bibitem[Coo05]{Coo05}
U.~Coope.
\newblock {\em Time for Aristotle: Physics IV. 10--14}.
\newblock Oxford Aristotle Studies. Oxford Univ.\ Press, New York, 2005.

\bibitem[Cow99]{Cow99}
H.~Coward.
\newblock Time in {H}induism.
\newblock {\em J.\ Hindu-Christian Studies}, 12:22--27, 1999.

\bibitem[DG15]{DewGra15}
B.S. DeWitt and N.~Graham, editors.
\newblock {\em {The Many Worlds Interpretation of Quantum Mechanics}}.
\newblock Princeton Univ.\ Press, 2015.

\bibitem[DHP03]{DenHiePru03}
R.\ Denk, M.\ Hieber, and J.\ Prüss.
\newblock {\em ${{R}}$-{B}oundedness, {F}ourier {M}ultipliers and {P}roblems of
  {E}lliptic and {P}arabolic {T}ype}, volume 788 of {\em Mem.\ Amer.\ Math.\
  Soc.}
\newblock Amer.\ Math.\ Soc., Providence, RI, 2003.

\bibitem[EK17]{EitKye17}
T.~Eiter and M.~Kyed.
\newblock Time-periodic linearized {N}avier-{S}tokes equations: an approach
  based on {F}ourier multipliers.
\newblock In {\em Particles in flows}, Adv. Math. Fluid Mech., pages 77--137.
  Birkh\"{a}user, Cham, 2017.

\bibitem[EML16]{Laasri2016}
O.~El-Mennaoui and H.~Laasri.
\newblock On evolution equations governed by non-autonomous forms.
\newblock {\em Arch.\ Math.}, 107:43--57, 2016.

\bibitem[EN00]{EngNag00}
K.-J.\ Engel and R.\ Nagel.
\newblock {\em One-{P}arameter {S}emigroups for {L}inear {E}volution
  {E}quations}, volume 194 of {\em Graduate Texts in Mathematics}.
\newblock Springer-Verlag, New York, 2000.

\bibitem[GHK17]{HieberGaldi}
Giovanni~P. Galdi, Matthias Hieber, and Takahito Kashiwabara.
\newblock Strong time-periodic solutions to the 3{D} primitive equations
  subject to arbitrary large forces.
\newblock {\em Nonlinearity}, 30(10):3979--3992, 2017.

\bibitem[GHN16]{Hieber2016_per}
M.~Geissert, M.~Hieber, and Th.~H. Nguyen.
\newblock A general approach to time periodic incompressible viscous fluid flow
  problems.
\newblock {\em Arch. Ration. Mech. Anal.}, 220(3):1095--1118, 2016.

\bibitem[GTV12]{Git2012}
D.~M. Gitman, I.~V. Tyutin, and B.~L. Voronov.
\newblock {\em Self-adjoint extensions in quantum mechanics}, volume~62 of {\em
  Progress in Mathematical Physics}.
\newblock Birkh\"{a}user/Springer, New York, 2012.
\newblock General theory and applications to Schr\"{o}dinger and Dirac
  equations with singular potentials.

\bibitem[Haa06]{Haa06}
M.\ Haase.
\newblock {\em The {F}unctional {C}alculus for {S}ectorial {O}perators}, volume
  169 of {\em Oper.\ Theory Adv.\ Appl.}
\newblock Birkh{\"a}user, Basel, 2006.

\bibitem[HE73]{Hawking1973}
S.~W. Hawking and G.~F.~R. Ellis.
\newblock {\em The large scale structure of space-time}.
\newblock Cambridge University Press, London-New York, 1973.
\newblock Cambridge Monographs on Mathematical Physics, No. 1.

\bibitem[Hie20]{Hieber2020}
Matthias Hieber.
\newblock On operator semigroups arising in the study of incompressible viscous
  fluid flows.
\newblock {\em Philos. Trans. Roy. Soc. A}, 378(2185):618--639, 2020.

\bibitem[HKKT20]{Hieber2017b_per}
Matthias Hieber, Naoto Kajiwara, Klaus Kress, and Patrick Tolksdorf.
\newblock The periodic version of the {D}a {P}rato--{G}risvard theorem and
  applications to the bidomain equations with {F}itz{H}ugh-{N}agumo transport.
\newblock {\em Ann. Mat. Pura Appl. (4)}, 199(6):2435--2457, 2020.

\bibitem[HM00]{HieMon00}
M.\ Hieber and S.~Monniaux.
\newblock Heat-kernels and maximal ${L}^p-{L}^q$-estimates: The non-autonomous
  case.
\newblock {\em J.\ Fourier Anal.\ Appl.}, 6:467--481, 2000.

\bibitem[HM13]{HusMug13}
A.\ Hussein and D.\ Mugnolo.
\newblock Quantum graphs with mixed dynamics: the transport/diffusion case.
\newblock {\em J.\ Phys.\ A}, 46:235202, 2013.

\bibitem[HMT19]{Hieber2019_per}
M.~Hieber, A.~Mahalov, and R.~Takada.
\newblock Time periodic and almost time periodic solutions to rotating
  stratified fluids subject to large forces.
\newblock {\em J. Differential Equations}, 266(2-3):977--1002, 2019.

\bibitem[HNS17]{Hieber2017_per}
M.~Hieber, Th.~H. Nguyen, and A.~Seyfert.
\newblock On periodic and almost periodic solutions to incompressible viscous
  fluid flow problems on the whole line.
\newblock In {\em Mathematics for nonlinear phenomena---analysis and
  computation}, volume 215 of {\em Springer Proc. Math. Stat.}, pages 51--81.
  Springer, Cham, 2017.

\bibitem[HP97]{HiePru97}
M.~Hieber and J.~Pr\"uss.
\newblock Heat kernels and maximal ${L}^p-{L}^q$ estimates for parabolic
  evolution equations.
\newblock {\em Comm.\ Partial Differ.\ Equations}, 22:1647--1669, 1997.

\bibitem[HS20]{HieberStinner}
Matthias Hieber and Christian Stinner.
\newblock Strong time periodic solutions to {K}eller-{S}egel systems: an
  approach by the quasilinear {A}rendt-{B}u theorem.
\newblock {\em J. Differential Equations}, 269(2):1636--1655, 2020.

\bibitem[JZ12]{JacZwa12}
B.~Jacob and H.~Zwart.
\newblock {\em Linear Port-Hamiltonian Systems on Infinite-dimensional Spaces},
  volume 223 of {\em Oper.\ Theory Adv.\ Appl.}
\newblock Birkh{\"a}user, Basel, 2012.

\bibitem[Kat80]{Kat80}
T.\ Kato.
\newblock {\em {Perturbation Theory for Linear Operators}}, volume 132 of {\em
  Grundlehren der mathematischen Wissenschaften}.
\newblock Springer-Verlag, Berlin, 1980.

\bibitem[KS17]{KyeSau17}
M.~Kyed and J.~Sauer.
\newblock A method for obtaining time-periodic {$L^p$} estimates.
\newblock {\em J. Differential Equations}, 262:633--652, 2017.

\bibitem[KW01]{KalWei01}
N.J. Kalton and L.~Weis.
\newblock The ${H}^\infty$-calculus and sums of closed operators.
\newblock {\em Math.\ Ann.}, 321:319--345, 2001.

\bibitem[KW04]{KunWei04}
P.C. Kunstmann and L.~Weis.
\newblock Maximal {$L_p$}-regularity for parabolic equations, {F}ourier
  multiplier theorems and {$H^\infty$}-functional calculus.
\newblock In {\em {Functional Analytic Methods for Evolution Equations}},
  volume 1855 of {\em Lect.\ Notes Math.}, pages 65--311. Springer-Verlag,
  Berlin, 2004.

\bibitem[Mar12]{Mar12}
S.~Martin.
\newblock Time, kingship, and the {M}aya universe.
\newblock {\em Expedition}, 54:18--23, 2012.

\bibitem[Mug14]{Mug14}
D.~Mugnolo.
\newblock {\em {Semigroup Methods for Evolution Equations on Networks}}.
\newblock Underst.\ Compl.\ Syst. Springer-Verlag, Berlin, 2014.

\bibitem[PS16]{PruessSimonett}
J.~Pr\"{u}ss and G.~Simonett.
\newblock {\em Moving interfaces and quasilinear parabolic evolution
  equations}, volume 105 of {\em Monographs in Mathematics}.
\newblock Birkh\"{a}user/Springer, Cham, 2016.

\bibitem[Rov18]{Rov18}
C.~Rovelli.
\newblock {\em The Order of Time}.
\newblock Penguin Random House, New York, 2018.

\bibitem[Vej82]{Vej82}
O.~Vejvoda.
\newblock {\em {Partial differential equations: time-periodic solutions}}.
\newblock Martinus Nijhoff Publishers, The Hague, 1982.

\end{thebibliography}

\end{document}